\documentclass{amsart}

\usepackage[pdftex]{graphicx}

\usepackage{amsmath}
\usepackage{amssymb}

\usepackage{bm}

\usepackage{amscd}
\usepackage{multicol}
\usepackage{MnSymbol}
\usepackage{tikz}
\usepackage[OT2,T1]{fontenc}
\usepackage{mathrsfs}
\usepackage{comment} 
\usepackage{stmaryrd}
\usepackage[all]{xy}
\usepackage{longtable}
\usepackage{multirow,bigdelim}
\usepackage{amsthm}
\usepackage{mathrsfs}
\usepackage{bm}
\usepackage{color}
\usepackage{manfnt}
\usepackage{indentfirst}
\usepackage{autobreak}
\usetikzlibrary{arrows.meta} 
\usepackage{mathtools}
\mathtoolsset{showonlyrefs=true}
\allowdisplaybreaks[4]
\newcommand{\D}{\partial}
\newcommand{\mb}{\mathbb}
\newcommand{\ol}{\overline}

\usepackage{tikz}
    \usetikzlibrary{arrows.meta} 
\usepackage{tikz-cd}
\usetikzlibrary{cd}

 \DeclareMathOperator{\GL}{GL}
\DeclareMathOperator{\Mat}{Mat} 
\DeclareMathOperator{\End}{End}

 \DeclareMathOperator{\im}{Im}
\DeclareMathOperator{\trdeg}{tr\text{.}deg} \DeclareMathOperator{\wt}{wt}
  \DeclareMathOperator{\Span}{Span}

\DeclareMathOperator{\dep}{dep}

\title[Hyperderivatives of the deformation series]{Hyperderivatives of the deformation series associated with arithmetic gamma values and characteristic $p$ multiple zeta values}
\author{Ryotaro Harada and Daichi Matsuzuki}
		\address{
				Tohoku University, Aza-Aoba 6-3 Aramaki, Aoba-ku, Sendai City, Miyagi, 980-8578, Japan.}
        \address{
				National Tsing Hua University, 101, Section 2, Kuang-Fu Road, Hsinchu 300, Taiwan R.O.C.}
			\email{harada@tohoku.ac.jp}
            \email{matsuzuki@math.nthu.edu.tw}
\date{October 30, 2024}

\newtheorem{thm}{Theorem}[section]
\newtheorem{lem}[thm]{Lemma}
\newtheorem{cor}[thm]{Corollary}
\newtheorem{prop}[thm]{Proposition}
\theoremstyle{remark}
        
\theoremstyle{definition}
\newtheorem{defn}[thm]{Definition}

\newtheorem{eg}[thm]{Example}

\begin{document}
\bibliographystyle{amsalpha+}

\begin{abstract}
In the number theory in positive characteristic, there are analogues of some special values introduced by Carlitz, Carlitz gamma values and Carlitz zeta values for instance. Each of them is further developed to arithmetic gamma values and multiple zeta values by Goss and Thakur respectively.  
In this paper, by generalizing a result of Chang-Papanikolas-Thakur-Yu (2010), we obtain the algebraic independence of certain arithmetic gamma values, positive characteristic multiple zeta values of restricted indices and hyperderivatives of their deformations. 
We prove this by using Chang-Papanikolas-Yu's derivation, Maurichat's prolongation, Namoijam's formula and Papanikolas' theory of $t$-motivic Galois group. 
\end{abstract}


\maketitle
\tableofcontents
\setcounter{section}{-1}

\section{Introduction}
\subsection{Background in positive characteristic}
Let $A:=\mathbb{F}_q[\theta]$ be the polynomial ring in the variable $\theta$ over finite field $\mathbb{F}_q$, with quotient field $k:=\mathbb{F}_q(\theta)$, and $A_{+}$ be the set of monic polynomials in $A$.
Let $k_{\infty}$ be the completion of $k$ at the infinite place $\infty$, and $\mathbb{C}_{\infty}$ be the completion of a fixed algebraic closure $\overline{k_{\infty}}$ at $\infty$.
In this paper, we deal with positive characteristic analogues of gamma values, multiple zeta values, and deformation series associated to them.  

Let us briefly recall the pioneering works of Carlitz (\cite{Ca35} for example) on periods in positive characteristic.
He introduced the Carlitz period $\tilde{\pi}$ (an analogue of $2\pi\sqrt{-1}$), Carlitz gamma values $\Gamma_C(s)$ (analogues of gamma values, $ s \geq 1$), and Carlitz zeta values $\zeta_A(s)$ (analogues of Riemann zeta values, $ s \geq 1$). It was proven by him that these values satisfy analogue of Euler's famous formula for Riemann zeta values $\zeta(2s)\ (s\geq 1)$.

Algebraic independence property of these analogues are well studied.
Let $s=\sum_{j\geq 0}s_jq^j \in \mathbb{N}$ with $0\leq s_j\leq q-1$. Then Carlitz gamma values are defined by $\Gamma_C(s):=\prod_i\bigl(\prod_{j=0}^{i-1}(\theta^{q^i}-\theta^{q^j} )\bigr)^{s_i}\in A_+$.
Goss (\cite{G80}) gave a generalization of Carlitz gamma value, so-called the arithmetic gamma value. For $s=\sum_{j\geq 0}s_jq^j\in\mathbb{Z}_p$ with $0\leq s_j\leq q-1$, the arithmetic gamma values are defined as the following form:
\[
    \Gamma(s):=\prod_i\biggl(\prod_{j=0}^{i-1}\Bigl( 1-\frac{\theta^{q^j}}{\theta^{q^i}} \Bigr)\biggr)^{s_i}\in k_{\infty}.
\]
Goss' definition of gamma values is essentially the same as the Carlitz gamma values when $s \in \mathbb{N}$.

In 2010, Chang, Papanikolas, Thakur and Yu determined the algebraic independence among arithmetic gamma values and Carlitz zeta values (\cite{CPTY10}) as the following:
\begin{thm}[{\cite[Theorem 4.2.2]{CPTY10}}]\label{CPTY10}
    Given any $s, l\in\mathbb{N}$, let $E$ be the field generated by the set 
    \[
        \{\tilde{\pi}, \zeta_A(1), \ldots, \zeta_A(s)  \}\cup \{ \Gamma\bigl({c/(1-q^l)}\bigr)\ |\ 1\leq c\leq q^l-2 \}.
    \]
    Then the transcendence degree of $E$ over $\overline{k}$ is
    \[
        s-\left\lfloor \frac{s}{p} \right\rfloor-\left\lfloor \frac{s}{q-1} \right\rfloor+\left\lfloor \frac{s}{p(q-1)} \right\rfloor + l.
    \]
\end{thm}
They obtained this result by using Papanikolas' theory of $t$-motivic Galois group which we will recall later in Section \ref{Preliminaries}.
As Chang and Yu~(\cite{CY07}) had determined all algebraic relations among Carlitz zeta values and Carlitz period, the contributions of the paper \cite{CPTY10} on transcendence are the following two: 1. determing all the algebraic relations among gamma values, and 2. proving that we have no algebraic relations between Carlitz zeta values and arithmetic gamma values except for those come from the Carlitz period.

The aim of this paper is to generalize Theorem \ref{CPTY10} in two directions.
First, we deal with Thakur's multiple zeta values, generalizations of Carlitz zeta values with multi-index.
Second, we consider certain deformations of Carlitz periods, arithmetic gamma values, and multiple zeta values. 

Thakur~(\cite{T04}) defined multiple zeta values in positive characteristic (MZVs in short) by the following form for each index $(s_1, \ldots, s_n)\in\mathbb{N}^n$ ($\mathbb{N}$ is the set of positive integers):
\[
	\zeta_A(s_1, \ldots, s_d):=\sum_{\substack{\deg a_1>\cdots >\deg a_d\geq 0\\ a_1, \ldots, a_d\in A_{\plus}}}\frac{1}{a_1^{s_1}\cdots a_d^{s_d}}\in k_{\infty}.
\]
For each ${\bf s}=(s_1, \ldots, s_d)\in\mathbb{N}^d$, we define $\wt({\bf s})$ (weight of ${\bf s}$) and $\dep({\bf s})$ (depth of ${\bf s}$) by $\wt({\bf s}):=s_1+\cdots +s_d$ and $\dep({\bf s}):=d$ respectively. 
The Carlitz zeta values are no other than multiple zeta values with indices of depth $1$.


We can consider the lifting of the aforementioned special values to deformation series. Let $t$ be an independent variable of $\theta$.
For $n\in\mathbb{Z}$, we consider the automorphism of the Laurent series field $\mathbb{C}_{\infty}((t))$, which is called Frobenius $n$-twisting:
\begin{align*}
\mathbb{C}_{\infty}((t))&\rightarrow \mathbb{C}_{\infty}((t))\\
		f:=\sum_i a_it^i&\mapsto f^{(n)}:=\sum_i a_i^{q^n}t^i.
\end{align*}  

For each $l\geq 1$, we fix the $(q^l-1)$th root of $-\theta$. In \cite{CPTY10}, they defined
$$
\Omega_l:=(-\theta)^{\frac{-q^l}{q^l-1}}\prod_{i=1}^{\infty}\biggl(1-\frac{t}{\theta^{q^{il}}}\biggr)\in\mathbb{E},
$$
where
\begin{equation}
    \mathbb{E}= \left\{\sum_{i=0}^{\infty}a_it^i\in\overline{k}\llbracket t \rrbracket \,\middle| \,\lim_{i\to\infty}|a_i|_\infty^{1/i}=0,~[k_\infty(a_0,a_1,\dots):k_\infty]<\infty\right\},
\end{equation}
the ring of entire functions.
When $t=\theta$, $\Omega:=\Omega_1$ becomes $\tilde{\pi}^{-1}$.
Inspired by \cite{CPTY10}, we define a deformation series of arithmetic gamma values for $s=\sum_{i\geq 0}s_iq^i\in\mathbb{Z}_p$ with $0\leq s_j\leq q-1$:
$$
G(s):=\prod_i\biggl(\prod_{j=0}^{i-1}\Bigl( 1-\frac{t^{q^j}}{\theta^{q^i}} \Bigr)\biggr)^{s_i}\in \mathbb{T}.
$$
Here, $\mathbb{T}$ is the subring of $\mathbb{C}_{\infty}\llbracket t \rrbracket$ consisting of power series convergent on the closed unit disc $|t|_{\infty}\leq 1$.
By the definition, we get 
$$
G(s)|_{t=\theta}=\Gamma(s).
$$
Also, Anderson and Thakur defined that of MZVs, which is known as Anderson-Thakur series (see \cite{AT09} and \cite{CGM21}), as the following form:
$$
\zeta^{\rm AT}_A(s_1, \ldots, s_d):=\sum_{i_1>\cdots >i_d\geq 0}\frac{H_{s_1-1}^{(i_1)}\cdots H_{s_d-1}^{(i_d)}}{\mathbb{L}_{i_1}^{s_1}\cdots \mathbb{L}_{i_d}^{s_d}}\in \mathbb{T}
$$
where $\mathbb{L}_{i_j}:=(t-\theta)^{(1)}\cdots(t-\theta)^{(i_j)}$ and $H_{s_i-1}\in A[t]$ is Anderson-Thakur polynomial (\cite{AT90}). 
In \cite{AT09}, the following is proved:
$$
\zeta^{\rm AT}_A(s_1, \ldots, s_d)|_{t=\theta}=\Gamma_{C}(s_1)\cdots\Gamma_{C}(s_d)\zeta_A(s_1, \ldots, s_d).
$$

We deal not only with special values at $t=\theta$ of the aforementioned power series, which are essentially the same as Carlitz period, arithmetic gamma values, and MZVs, but also with special values of a kind of derivatives of those power series.
In positive characteristic, there are several analogues of derivatives. In this paper, we use hyperderivative, which is defined as follows:
\begin{defn}\label{defnhyperderivative}
For each $n \geq0$, we define the $\mathbb{C}_\infty$-linear operator $\partial^{(n)}_t$ on $\mathbb{C}_\infty((t))$, which is called \textit{$n$-th hyperderivative}, by
\begin{align}\label{defHD}
\partial_t^{(n)}:\mathbb{C}_\infty((t))&\rightarrow\mathbb{C}_\infty((t))\\
\sum_{i=m}^\infty a_i t^i&\mapsto\sum_{i=m}^\infty \binom{i}{n} a_i t^{i-n}\nonumber 
\end{align}
where $a_i \in\mathbb{C}_\infty$, $\binom{i}{n}$ are the binomial coefficients in $\mathbb{F}_p$, and we put $\binom{i}{n}=0$ for $i <n$.
\end{defn}
In \cite{M18}, Maurischat invented the extension of $t$-motive via hyperderivative, which is so-called prolongation. This allows us to show algebraic independence among hyperderivatives of deformation series in Tate algebra, by computing the dimension of $t$-motivic Galois group associated to the pre-$t$-motive under the prolongation.  

We pick a non-negative integer $n$ and a finite set $I$ of indices such that 
        \begin{equation}
            I=\bigcup_{\mathbf{s}\in I} \operatorname{Sub}(\mathbf{s}),\label{subclosed}
        \end{equation}
        where $\operatorname{Sub}(\mathbf{s})$ is defined to be the set $\{ (s_i,\,\dots,\,s_j) \mid 1 \leq i \leq j \leq d\}$ for an index $\mathbf{s}=(s_1,\,\dots,\,s_d)$. 
        Our main result is the following:
\begin{thm}\label{intro main result}
For $n\geq 0 $ and $l \geq 1$ and a finite set $I$ of indices such that \eqref{subclosed} holds, we have
    \begin{align}
        &\trdeg_{k}k\left(\begin{matrix}\partial_{t}^{(n^\prime)}\Omega|_{t=\theta},\,\partial_{t}^{(n^\prime)}\Omega_l^{(-l^\prime)}|_{t=\theta},\\ \partial_{t}^{(n^\prime)}\zeta_{A}^{\rm{AT}}(\mathbf{s})|_{t=\theta} \end{matrix} \middle|\, 0\leq n^\prime \leq n,\,0\leq l^\prime\leq l-1,\,\mathbf{s} \in I
        \right)\\
        =&\trdeg_{k}k\left(\partial_{t}^{(n^\prime)}\Omega|_{t=\theta},\,\partial_{t}^{(n^\prime)}\Omega_l^{(-l^\prime)}|_{t=\theta} \, \middle|\, 0\leq n^\prime \leq n,\,0\leq l^\prime\leq l-1
        \right)\\
        &+\trdeg_{k}k\left(\partial_{t}^{(n^\prime)}\Omega|_{t=\theta},\,\partial_{t}^{(n^\prime)}\zeta_{A}^{\rm{AT}}(\mathbf{s})|_{t=\theta} \, \middle|\, 0\leq n^\prime \leq n,\,\mathbf{s} \in I
        \right)\\
        &-\trdeg_{k}k\left(\partial_{t}^{(n^\prime)}\Omega|_{t=\theta} \, \middle|\, 0\leq n^\prime \leq n
        \right).
    \end{align}
\end{thm}
For any finite set $I^\prime$ of indices, we can take a finite set $I$ containing $I^\prime$ such that \eqref{subclosed} holds. 
Therefore the theorem asserts the algebraic relations concerning $\Omega$ and its hyperderivatives at $t=\theta$ are only algebraic relations between hyperderivatives of $\Omega_l$ and those of Anderson-Thakur series at $t=\theta$.
The first and third terms of the right-hand side are calculated in \eqref{hyptran:Omega_1:and:omega_l} and \cite[Proof of Theorem 2.1]{M18}, respectively. 
About the second term, we have the following result. Calculating the second term in general cases seems to be a difficult problem.
\begin{thm}[{\cite[Theorem 5.19 (2)]{Ma24}}]\label{matsuzuki thm}
   If we take an index $\mathbf{s}=(s_1,\,\dots,\,s_r)\in \mathbb{Z}_{\geq 1}^{r}$ satisfying that $s_1,\,\dots,\,s_r$ are distinct, and $p \nmid s_i,\,(q-1) \nmid s_i$ for $1\leq i\leq r$, then the field extension over $\overline{k}$ generated by
    \begin{equation}
        \{ \partial_{t}^{(n^\prime)}\Omega|_{t=\theta},\,\partial_{t}^{(n^\prime)}\zeta_{A}^{\rm{AT}}(\mathbf{s})|_{t=\theta} \,  \mid 0\leq n^\prime  \leq n,\,\mathbf{s}^\prime \in \operatorname{Sub}^\prime(\mathbf{s}) \}
    \end{equation}
    has transcendental degree $(n+1)\left(\# \operatorname{Sub}^\prime(\mathbf{s}) +1\right)=(n+1)(2^r)$ over $\overline{k}$ for any integer $n \geq 0$. Equivalently, the set above is algebraically independent over $\overline{k}$.
\end{thm}

\subsection{Background in characteristic 0}
In classical case, there are attempts to get the functions which do not satisfy any algebraic differential equations, so-called the hypertranscendental functions. In 1887, H\"{o}lder (\cite{H87}) proved that Euler gamma function is a hypertranscendental function over rational function field. Later in 1901, Hilbert (\cite{Hi01}) proved the hypertranscendence of the Riemann zeta function by using the functional equation among zeta and gamma functions. Bank and Kaufman (\cite{BK76}) extended their work to the case of hypertranscendence over the field of meromorphic functions. Lately, Huang and Wai Ng (\cite{HWN20}) further shown the hypertranscendence of certain functions involving meromorphic functions and exponential functions. There is a generalization of Riemann zeta function, known as the multiple zeta function (see \cite{Mat02} for example). It is natural to ask the hypertranscendence of multiple zeta functions but we do not know that they satisfy or not.

This paper is organized as follows. In section 1, we recall fundamental notations. In section 2, we recall the prolongation and the derivation of pre-$t$-motive. In section 3, we prove the algebraic independence result among arithmetic gamma values and hyperderivatives of their deformation series with $t=\theta$, by applying Namoijam's formula. In sections 4 and 5, we construct the prolongations of pre-$t$-motives associated to arithmetic gamma values and multiple zeta values. In section 6, we determine the dimension of $t$-motivic Galois group and clarify, by using Papanikolas' theory (\cite{P08}), the transcendence degree of the field generated by deformation series of arithmetic gamma values, multiple zeta values, and their hyperderivatives.      

\section{Preliminaries}\label{Preliminaries}
\subsection{Notations}
\label{No}
We use the following notations.

\begin{itemize}
\setlength{\leftskip}{1.0cm}
\item[$\mathbb{N}=$] the set of positive integers.
\item[$q=$] a power of a prime number $p$.  
\item[$\mathbb{F}_q=$] a finite field with $q$ elements.
\item[$\theta$, $t=$] independent variables.
\item[$A=$] the polynomial ring $\mathbb{F}_q[\theta]$.
\item[$A_{+}=$] the set of monic polynomials in $A$.
\item[$A_{d+}=$] the set of elements in $A_{+}$ of degree $d$. 
\item[$k=$] the rational function field $\mathbb{F}_q(\theta)$.
\item[$k_{\infty}=$] the completion of $k$ at the infinite place $\infty$, $\mathbb{F}_q((\frac{1}{\theta}))$.
\item[$\overline{k_{\infty}}=$] a fixed algebraic closure of $k_{\infty}$.
\item[$\mathbb{C}_{\infty}=$] the completion of $\overline{k_{\infty}}$ at the infinite place $\infty$.
\item[$\overline{k}=$] a fixed algebraic closure of $k$ in $\mathbb{C}_{\infty}$.
\item[$k^{\rm sep}=$] a fixed separable closure of $k$ in $\overline{k}$.
\item[$|\cdot|_{\infty}=$] a fixed absolute value for the completed field $\mathbb{C}_{\infty}$ so that $|\theta|_{\infty}=q$.
\item[$\mathbb{T}=$] the Tate algebra over $\mathbb{C}_{\infty}$, the subring of $\mathbb{C}_{\infty}\llbracket t \rrbracket$ consisting of power series convergent on the closed unit disc $|t|_{\infty}\leq 1$.
\item[$\mathbb{L}=$] the quotient field of $\mathbb{T}$.
\item[$\mathbb{E}=$] $\{\sum_{i=0}^{\infty}a_it^i\in\overline{k}\llbracket t \rrbracket\mid \lim_{i\to\infty}|a_i|_\infty^{1/i}=0,~[k_\infty(a_0,a_1,\dots):k_\infty]<\infty\}.$
\end{itemize}


For nonzero elements $x, y\in\mathbb{C}_{\infty}((t))$, we write $x\sim y$ when $x/y\in\ol{k}$.


\begin{defn}[\cite{CPY11}]\label{defn:higher level pre t-motive}
For $r\in\mathbb{N}$, let $\overline{k}[\sigma^{r}, \sigma^{-r}]$ be the non-commutative ring of Laurent polynomials in $\sigma^{r}$ with coefficients in $\overline{k}(t)$, subject to the relation
\[
	\sigma^rf:=f^{(-r)}\sigma^r\ \text{for all $f\in\overline{k}(t)$}.
\]
 A {\it pre-$t$-motive of level $r$} is a left $\overline{k}(t)[\sigma^r, \sigma^{-r}]$-module $M$ that is finite dimensional over $\overline{k}(t)$.
\end{defn}

Once we fix a $\overline{k}(t)$-vector space $M$ of rank $d$ with a fixed $\overline{k}(t)$-basis ${\bf m}=(m_1, \ldots, m_d)^{tr}$ and $\Phi\in \GL_{d}(\overline{k}(t))$, we can uniquely determine the pre-$t$-motive structure on $M$ by setting $\sigma^r {\bf m} = \Phi {\bf m}$ (cf. \cite[\S 3.2.3]{P08}). In this case, we call $M$ the pre-$t$-motive defined by $\Phi$.

For a pre-$t$-motive $P$ of level $r$, we put
$
    P^{\text{B}}:=\{a \in \mathbb{L}\otimes_{\overline{k}(t)}P \mid \sigma(a)=a \}
$
where the $\sigma^r $-action on $\mathbb{L}\otimes_{\overline{k}(t)}P$ is given by $\sigma(f\otimes m):=f^{(-r)}\otimes \sigma^r m$ for $f \in \mathbb{L}$ and $m \in P$, following \cite{P08}. 
We can show that $P^{\text{B}}$ is an $\mathbb{F}_{q^r}(t)$-vector space, and call it the \textit{Betti realization} of $P$.
We say that $P$ is \textit{rigid analytically trivial}, if the natural map 
\begin{equation}
\mathbb{L}\otimes_{\mathbb{F}_{q^r}(t)}P^{\text{B}} \rightarrow \mathbb{L}\otimes_{\overline{k}(t)}P
\end{equation}
is an isomorphism of $\mathbb{L}$-vector spaces.
Papanikolas obtained the following criterion for rigid analytic triviality of pre-$t$-motives, see also \cite{A86}.

\begin{prop}[{\cite[Theorem 3.3.9]{P08}}]
    Let $P$ be a pre-$t$-motive of level $r$ and dimension $d$ over $\overline{k}(t)$ defined by $\Phi \in \operatorname{GL}_d(\overline{k}(t))$. For $P$ to be rigid analytically trivial, it is necessary and sufficient that there exists a matrix $\Psi \in { \rm GL}_d (\mathbb{L})$ such that $\Psi^{(-r)}=\Phi \Psi$.
\end{prop}

Let us take a rigid analytically trivial pre-$t$-motive $P$ defined by $\Phi$. The matrix $\Psi$ in the proposition above is called a \textit{rigid analytic trivialization} of $\Phi$. 
Rigid analytic trivializations of $\Phi$ is not unique. If $\Psi$ and $\Psi^\prime$ are two rigid analytic trivializations of a matrix $\Phi$, then $\Psi^{-1}\Psi^\prime \in { \rm GL}_d(\mathbb{F}_{q^r}(t))$ (\cite[\S 4.1]{P08}). 
Let us write $\Psi^{-1}=\Theta=(\Theta_{ij})$. In the case when an entry $\Theta_{ij}$ converges at $t=\theta$, the value $\Theta_{ij}|_{t=\theta}$ is called a \textit{period} of $P$ (cf.~\cite{P08}). The following proposition proves that the entries of the matrices $\Psi$ we consider in the following context are entire.

\begin{prop}[{\cite[Proposition 3.1.3]{ABP04}}] \label{ABP04-3.1.3}
 For a given matrix
        $ \Phi \in \operatorname{Mat}_{d \times d}(\overline{k}[t]),$
    suppose that there exists  
    $\psi \in \operatorname{Mat}_{d\times 1}(\mathbb{T})$
    so that $ \psi^{(-r)}=\Phi \psi$.
    If  $\det \Phi|_{t=\theta} \neq 0$, then all entries of $\psi$ are entire.
\end{prop}

The theory of Papanikolas (\cite{P08}) enable us to study the transcendence of periods of rigid analytically trivial pre-$t$-motives via Tannakian duality.
Papanikolas proved that rigid analytically trivial pre-$t$-motives form a neutral Tannakian category over $\mathbb{F}_q(t)$ (for the definition of Tannakian category, we refer the readers to \cite{DM82}).

\begin{thm}[{\cite[Theorem 3.3.15]{P08}}]
The category $\mathcal{R}$ of all rigid analytically trivial pre-$t$-motive of level $r$ forms a neutral Tannakian category over $\mathbb{F}_{q^r}(t)$ with the fiber functor $P\mapsto P^{\text{B}}$.
\end{thm}

For a rigid analytically trivial pre-$t$-motive $P$ of level $r$, we denote by $\langle P \rangle $ the Tannakian sub-category generated by $P$ in this paper. By Tannakian duality, there exists an algebraic group, say $\Gamma_P$, such that the category $\operatorname{Rep}_{\mathbb{F}_{q^r}(t)}(\Gamma_P)$ of finite-dimensional linear representations over $\mathbb{F}_{q^r}(t)$ of it is equivalent to $\langle P \rangle $.
We call the algebraic group $\Gamma_P$ the \textit{$t$-motivic Galois group} of $P$.



Suppose that matrices $\Phi\in \operatorname{GL}_d(\overline{k}(t))$ and $\Psi\in \operatorname{GL}_d(\mathbb{L})$ for which $\Psi^{(-r)}=\Phi \Psi$ are given. Then we put $\Psi_1:=(\Psi_{ij}\otimes1) \in \operatorname{GL}_d(\mathbb{L}\otimes_{\overline{k}(t)}\mathbb{L})$, $\Psi_2:=(1\otimes\Psi_{ij})\in \operatorname{GL}_d(\mathbb{L}\otimes_{\overline{k}(t)}\mathbb{L})$, and $\tilde{\Psi}:=\Psi_1^{-1}\Psi_2$. Let us consider the algebraic sub-variety 
\begin{equation}
    \Gamma_{\Psi}:=\operatorname{Spec}\mathbb{F}_{q^r}(t)[\tilde{\Psi}_{ij},\,1/\det\tilde{\Psi}] \label{GammaPsi}
\end{equation}
 of $\operatorname{GL}_{d/\mathbb{F}_{q^r}(t)}$ over $\mathbb{F}_{q^r}(t)$, the smallest closed subscheme of $\operatorname{GL}_{d/\mathbb{F}_q(t)}$ which has $\tilde{\Psi}$ as its $\mathbb{L}\otimes_{\overline{k}(t)} \mathbb{L}$-valued point. 
 We use this variety in order to explicitly determine the $t$-motivic Galois group $\Gamma_P$ in the case where $\Psi$ is a rigid analytic trivialization of a pre-$t$-motive $P$. 
The following theorem is Chang's refinement of the Papanikolas' theorem (see \cite{C09} and \cite{P08}), which claims that the variety $\Gamma_{\Psi}$ is isomorphic to the $t$-motivic Galois group of a pre-$t$-motive $P$ if $\Psi$ is a  rigid analytic trivialization of $\Phi$ defining $P$ and has a connection with transcendence theory.

\begin{thm}[{\cite{P08}, \cite[(1.2)]{C09}}]\label{RefinedPapanikolasThm}
The scheme $\Gamma_{\Psi}$ is a closed algebraic subgroup of $\GL_n/\mathbb{F}_{q^r}(t)$, which is isomorphic to the Galois group $\Gamma_M$ over $\mathbb{F}_{q^r}(t)$. Moreover, $\Gamma_{\Psi}$ has the following properties:
\begin{enumerate}
    \item [(i)] $\Gamma_{\Psi}$ is smooth over $\ol{\mathbb{F}_{q^r}(t)}$ and is geometrically connected.
    \item [(ii)] $\dim\Gamma_{\Psi}=\trdeg_{\overline{k}(t)} \overline{k}(t)(\Psi)=\trdeg_{\overline{k}} \overline{k}(\Psi|_{t=\theta})$.
\end{enumerate}
\end{thm}

Let $P$ and $P^\prime$ be rigid analytically trivial pre-$t$-motives defined respectively by the matrices $\Phi \in \operatorname{GL}_d(k(t))$ and $\Phi^\prime \in \operatorname{GL}_{d^\prime}(k(t))$ with $\Psi\in \operatorname{GL}_d(\mathbb{L})$ and $\Psi^\prime\in \operatorname{GL}_{d^\prime}(\mathbb{L})$ as their rigid analytic trivializations.
As the Tannakian categories $\langle P \rangle $ and $\langle P^\prime \rangle $ can be seen as subcategories of the category $\langle P \oplus P^\prime \rangle $ generated by their direct product, Tannakian duality yields homomorphisms $\pi : \Gamma_{P\oplus P^\prime}\twoheadrightarrow \Gamma_{P}$ and $\pi^\prime : \Gamma_{P\oplus P^\prime}\twoheadrightarrow \Gamma_{P^\prime}$ of algebraic groups. These homomorphisms are faithfully flat (\cite[Proposition 2.21]{DM82}) and hence these induce surjective homomorphisms of groups of $\overline{\mathbb{F}_{q^r}(t)}$-valued points.
We can describe these homomorphisms in terms of identifications $\Gamma_P=\Gamma_\Psi$ and $\Gamma_{P^\prime}=\Gamma_{\Psi^\prime} $ by Lemma \ref{LemmaTannakianProj} below. Although this is well-known by experts, we give a short proof for the present paper to be self-contained.

We note that the direct sum $P\oplus P^\prime$ is defined by the matrices $\Phi\oplus\Phi^\prime$
and that this matrix has a rigid analytic trivialization   $\Psi\oplus\Psi^\prime$. 
 Throughout this paper, for any square matrices $B_1$ and $B_2$ the symbol $B_1\oplus B_2$ denotes the canonical block diagonal matrix
 \begin{equation}
    \begin{pmatrix}
        B_1& O\\
        O & B_2
    \end{pmatrix}.
\end{equation}
By the definitions of groups $\Gamma_{\Psi}$, $\Gamma_{\Psi^\prime}$, and $\Gamma_{\Psi\oplus\Psi^\prime}$, the algebraic group $\Gamma_{\Psi\oplus\Psi^\prime}$ is a closed subgroup of 
\begin{equation}
    \Gamma_{\Psi}\times\Gamma_{\Psi^\prime}=\left \{ B_1\oplus B_2 \ \middle| \ B_1\in \Gamma_{\Psi},\,B_2 \in   \Gamma_{\Psi^\prime} \right\}.
\end{equation}

\begin{lem}\label{LemmaTannakianProj}
In the notations as above, the following diagram commutes:
\begin{equation}
\begin{tikzpicture}[auto]
 \node (33) at (3, 3) {$\Gamma_{P\oplus P^\prime}$};
\node (00) at (0, 0) {$\Gamma_{P}$}; \node (30) at (3, 0) {$\Gamma_{P}\times\Gamma_{ P^\prime}$}; \node (60) at (6, 0) {$\Gamma_{P^\prime}$.};

\draw[->>] (33) to node {$\pi$}(00);
\draw[->>] (33) to node {$\pi^\prime$}(60);
\draw[->>] (30) to node {$\operatorname{pr}_1$}(00);
\draw[->>] (30) to node {$\operatorname{pr}_2$}(60);

\draw[{Hooks[right]}->] (33) to node {} (30);

\label{LemmaDiagram}
\end{tikzpicture}
\end{equation}
(See \cite[Example 2.3]{Mi15} for example.)
\end{lem}

\begin{proof}
The equivalence $\langle {P\oplus P^\prime} \rangle \simeq \operatorname{Rep}_{\mathbb{F}_q(t)}(\Gamma_{P\oplus P^\prime})$ of categories yields the $\Gamma_{P\oplus P^\prime}(R)$-action on the Betti realization $R\otimes_{\mathbb{F}_{q^r}(t)}(P\oplus P^\prime)^B$ of the pre-$t$-motive $P\oplus P^\prime$ given by
\begin{equation}
    \Psi_{P\oplus P^\prime}^{-1}\mathbf{p}\rightarrow (\Psi_{P\oplus P^\prime}\gamma)^{-1}\mathbf{p},\quad \gamma \in  \Gamma_{P\oplus P^\prime}(R),
\end{equation}
where $\mathbf{p}$ is the $\overline{k}(t)$-basis of $P$ corresponding to $\Phi \oplus \Phi^\prime$ for any $\mathbb{F}_q(t)$-algebra $R$ (\cite[Theorem 4.5.3]{P08}).
The equivalence $\langle {P} \rangle \simeq \operatorname{Rep}_{\mathbb{F}_q(t)}(\Gamma_{P})$ also gives us 
 a $\Gamma_{P}(R)$-action on $R\otimes_{\mathbb{F}_q(t)}P^B$ in the similar way. 

We consider the two $\Gamma_{P\oplus P^\prime}$-actions on $P^B$. The first one is the sub-representation of the $\Gamma_{P\oplus P^\prime}(R)$-action $$\Gamma_{P\oplus P^\prime}(R) \curvearrowright R\otimes_{\mathbb{F}_q(t)}(P\oplus P^\prime)^B.$$  
This comes from $P \in \langle {P\oplus P^\prime} \rangle$ and the equivalence $\langle {P\oplus P^\prime} \rangle \simeq \operatorname{Rep}_{\mathbb{F}_q(t)}(\Gamma_{P\oplus P^\prime})$, hence corresponds to $\pi :\Gamma_{P\oplus P^\prime} \twoheadrightarrow \Gamma_{P}$.
The second one is induced by the action $\Gamma_{P} \curvearrowright R\otimes_{\mathbb{F}_{q^r}(t)}P^B$ via the surjection $\operatorname{pr}_1|_{\Gamma_{P\oplus P^\prime}}:\Gamma_{P\oplus P^\prime}\rightarrow \Gamma_{P}$, which corresponds to the morphism $\Gamma_{P\oplus P^\prime} \hookrightarrow \Gamma_{P}\times\Gamma_{ P^\prime} 
\overset{\rm{pr}_1}{\twoheadrightarrow}
\Gamma_{P}$.
We can see that these two are the same one. 
As $\langle P \rangle $ is generated by $P$, left triangle of Lemma \ref{LemmaTannakianProj} is commutative. The commutativity of the right triangle is proved by similar arguments.
\end{proof}

\section{Prolongations of pre-$t$-motives and derived pre-$t$-motives}
In this section, we review two manipulations on pre-$t$-motives: 
Maurischat's theory on prolongations and Chang-Papanikolas-Yu's theory on derivation. 
The former is used to obtain period interpretations the special values of hyperderivatives of rigid analytic trivializations. The latter is used to simplify the structure of pre-$t$-motives.

\subsection{Prolongations of pre-$t$-motives}\label{subsectionprolongation}



Maurischat (\cite{M18}) introduced the technique called prolongations. 
Let us mention that Namoijam and Papanikolas (\cite{NP22}) studied hyperderivatives of entries of rigid analytic trivialiations of matrices defining pre-$t$-motives. We review the prolongation in the language of pre-$t$-motives in this subsection.




Recalling the definition of the hyperderivatives (Definition \ref{defnhyperderivative}), we have the following analogue of the Leibniz rule called \textit{generalized Leibniz rule} (cf. \cite{M22}):
\begin{equation}
   \partial_t^{(n)}(f_1\cdots f_r)=\sum_{j_1+j_2+\cdots+j_s=n}\partial_t^{(j_1)}f_1\partial_t^{(j_2)}f_2\cdots \partial_t^{(j_r)}f_r \label{Leibniz}
\end{equation}
for $n \geq 0$ and $f_1,\,\dots,\,f_r \in \mathbb{C}_\infty((t))$.

For any matrix $B=(b_{ij})$ with entries in $\mathbb{C}_\infty((t))$, we define $\partial_t^{(n)}B:=(\partial_t^{(n)}b_{ij})$. 
Using equation \eqref{Leibniz}, we can show that the map \begin{equation}
    \rho_{n}:\operatorname{Mat}_{d\times d}\bigl(\mathbb{C}_\infty((t))\bigr) \rightarrow \operatorname{Mat}_{d(n+1)\times d(n+1)}\bigl(\mathbb{C}_\infty((t))\bigr)
\end{equation}
    defined by
\begin{equation}
    \rho_n(X):=
    \begin{pmatrix}
   X & 0 &\cdots&\cdots&0\\
   \partial_t^{(1)}X & X&0&\cdots&0\\
   \vdots& \vdots&\ddots&&\vdots \\
   \partial_t^{(n-1)}X & \partial_t^{(n-2)}X&\cdots&X&0\\
   \partial_t^{(n)}X &\partial_t^{(n-1)}X &\cdots&\partial_t^{(1)}X &X
\end{pmatrix}
\end{equation}
is a homomorphism of $\mathbb{C}_\infty((t))$-algebras for each $n \geq 0$ see \cite{M22}.


\begin{defn}
    For a pre-$t$-motive defined by a matrix $\Phi \in \operatorname{Mat}_r(\overline{k}[t])\cap \operatorname{GL}_r(\overline{k}(t))$, then its $n$-th prolongation $\rho_n P$ is the pre-$t$-motive defined by the matrix $\rho_n \Phi\in \operatorname{Mat}_{(n+1)r}(\overline{k}[t])\cap \operatorname{GL}_{(n+1)r}(\overline{k}(t))$.
\end{defn}
If we have an equation $\Psi^{(-1)}=\Phi \Psi$, then we have $(\rho_n\Psi)^{(-1)}=(\rho_n\Phi)\cdot( \rho_n\Psi)$ for each $n$ as $\rho_n$ is a homomorphism and commutes with $(-1)$-th fold twisting. Therefore, if $P$ is rigid analytically trivial then it follows that so is $\rho_n P$.

By the following proposition, the hyperderivative $\partial_t^{(i)}f$ of $f\in \mathbb{T}$ converge at $t=\theta$ if $f$ does.
\begin{prop}[{\cite[Corollary 2.7]{US98}}]\label{Uchino}
   Suppose $f \in \mathbb{T}$ converges at $t=\theta$ and we write
   \begin{equation}
  f=\sum_{i=0}^\infty a_i (t-\theta)^i\in \mathbb{T},
  \end{equation}
  then its hyperderivative $\partial_t^{(i)}f$ also converges at $t=\theta$ for each $i$ and we have $\partial_t^{(i)}f|_{t=\theta}=a_i$.
\end{prop}

\subsection{Derived pre-$t$-motives}\label{subsectionderived}
In \cite{CPY11}, they introduced the techniques of derived pre-$t$-motives, which we recall in this subsection.
In \cite{CPTY10}, they employed this technique to simplify the defining matrix of pre-$t$-motives which have gamma values as their periods.

\begin{defn}[{\cite{CPY11}}] 
\label{Defderived}
    For the pre-$t$-motives $P$ of level $r$ defined by a matrix $\Phi$ in $\operatorname{GL}_d(\overline{k}(t))$, its $s$-th derived pre-$t$-motive $P^{(s)}$ is defined to be the pre-$t$-motive of level $rs$ whose $\sigma^{rs}$-action is represented by the matrix
    \begin{equation}        
    \Phi^\prime:=\Phi^{(-sr+r)}\Phi^{(-sr+2r)}\cdots\Phi^{(-r)}\Phi \in \operatorname{GL}_d(\overline{k}(t)).
    \end{equation}
\end{defn}

Let $P$ be the pre-$t$-motive of level $r$ defined by a matrix $\Phi$ and let $\Psi$ be a rigid analytic trivialization of $\Phi$, that is, we have $\Psi^{(-r)}=\Phi \Psi$. 
We can show that the matrix $\Psi$ is also a rigid analytic trivialization of $\Phi^\prime:=\Phi^{(-sr+r)}\Phi^{(-sr+2r)}\cdots\Phi^{(-r)}\Phi$, which defines the $s$-th derived pre-$t$-motive $P^{(s)}$ of $P$ as follows:
 \begin{align}
    \Psi^{(-rs)}&=(\Psi^{(-r)})^{(-sr+r)}=(\Phi \Psi)^{(-sr+r)}\\
    &=(\Phi^{(-r)} \Psi^{(-r)})^{(-sr+2r)}=(\Phi^{(-r)}\Phi \Psi)^{(-sr+2r)}=\cdots=\Phi^\prime \Phi.
\end{align}
For any $s \geq 1$ and pre-$t$-motive $P$ of level $r$ whose defining matrix has rigid analytic trivialization $\Psi$, we have the equation:
\begin{equation} \label{derivedsameGaloisgroup}
    \Gamma_{P}=\Gamma_\Psi=\Gamma_{P^{(s)}}.
\end{equation}
of algebraic group over $\overline{\mathbb{F}_q(t)}$ by Theorem \ref{RefinedPapanikolasThm} of Papanikolas (more precisely, we have $\Gamma_{P^{s}}=\Gamma_{P} \times_{\operatorname{Spec}\mathbb{F}_{q^r}(t)}\operatorname{Spec}\mathbb{F}_{q^{rs}}(t)$).

\begin{eg}
The $s$-th derived Carlitz motive $C$ is defined by
\begin{equation}
   \Phi=\left((t-\theta^{(-s+1)})(t-\theta^{(-s+2)})\cdots(t-\theta^{(-1)})(t-\theta)\right)\in\GL_{1\times 1}(\ol{k}(t)).
\end{equation}
This has a rigid analytic trivialization $(\Omega)\in\GL_{1\times 1}(\mathbb{L})$.
\end{eg}

Let $P$ be a pre-$t$-motive and take non-negative integers $n$ and $s$. As the map $\rho_n$ commutes with Frobenius twisting and is a ring homomorphism, the pre-$t$-motives $\rho_n (P^{(s)})$ and $(\rho_n P)^{(s)}$ coincide with each other, that is, we can commute the prolongation and the derivation of pre-$t$-motives.

\section{Hypertranscendence of arithmetic gamma values}\label{Hypertranscendence of arithmetic gamma values}
In this section, we introduce the relations among deformation series $\Omega_l$ and $G(s)$, and Namoijam's formula in \cite{N21}. By using them we will prove algebraic independence among certain family of $G(s)$ and its hyperderivatives with $t=\theta$. We also prove algebraic independence of certain $\Gamma(s)$ and their hyperderivatives. 

We set $\mathbb{D}_i$ by
$$
    \mathbb{D}_i:=\begin{cases}
                    &(1-t/\theta^{q^i})(1-t/\theta^{q^{i-1}})^q\cdots (1-t/\theta^q)^{q^{i-1}}\quad \text{if $i\geq 1$,} \\
                    & 1\quad \text{if $i=0$}.
                  \end{cases}
$$


Recall that we write $x\sim y$ when $x/y\in\ol{k}$ for nonzero elements $x, y\in\mathbb{C}_{\infty}((t))$.
Then we can get the followings:

\begin{prop}\label{prop:periodpart}
    Fix $l\in\mathbb{Z}_{\geq 2}$. We have
    \[
        \frac{G(1/(1-q^l))}{G(q^{l-1}/(1-q^l))^{q}}\sim\Omega_l.
    \]
\end{prop}
\begin{proof}
  By $\mb{D}_i=(1-t/\theta^{q^i})\mb{D}^q_{i-1}$, we have
    \begin{align*}
        \frac{G(1/(1-q^l))}{G(q^{l-1}/(1-q^l))^{q}}=\frac{\prod_{i=0}^{\infty}\mb{D}_{il}}{\prod_{i=0}^{\infty}\mb{D}^{q}_{(i+1)l-1}}=\mb{D}_0\prod_{i=1}^{\infty}\frac{\mb{D}_{il}}{\mb{D}^{q}_{il-1}}=\prod_{i=1}^{\infty}\biggl(1-\frac{t}{\theta^{q^{il}}}\biggr)=(-\theta)^{\frac{q^l}{q^l-1}}\Omega_l.
    \end{align*}
\end{proof}

\begin{prop}\label{prop:quasiperiodpart}
    Fix $l\in\mathbb{Z}_{\geq 2}$. For each $1\leq j\leq l-1$, we have
    \[
        \frac{G(q^j/(1-q^l))}{G(q^{j-1}/(1-q^l))^{q}}\sim\Omega_l^{(-(l-j))}.
    \]
\end{prop}
\begin{proof}
By $\mb{D}_i=(1-t/\theta^{q^i})\mb{D}^q_{i-1}$, we get
\begin{align*}
\frac{G(q^j/(1-q^l))}{G(q^{j-1}/(1-q^l))^{q}}&=\frac{\prod_{i=0}^{\infty}\mb{D}_{il+j}}{\prod_{i=0}^{\infty}\mb{D}^{q}_{il+j-1}}=\prod_{i=0}^{\infty}\frac{\mb{D}_{il+j}}{\mb{D}^{q}_{il+j-1}}=\prod_{i=0}^{\infty}\biggl(1-\frac{t}{\theta^{q^{il+j}}}\biggr)=(-\theta)^{\frac{q^j}{q^l-1}}\Omega_l^{(-l+j)}.
\end{align*}
\end{proof}

When $t=\theta$, Proposition \ref{prop:periodpart} and Proposition \ref{prop:quasiperiodpart} recover \cite[Theorem 1.6]{T91} and \cite[Theorem 3.3.2]{CPTY10} respectively.

By using Proposition \ref{prop:periodpart} and \ref{prop:quasiperiodpart}, we can write $G(q^j/(1-q^l))$ by $\Omega_l^{(-(l-i))}\ (1\leq i\leq l)$ for each $j$ in the following way.
\begin{eg}
By combining Proposition \ref{prop:quasiperiodpart} of $j=l-1$ and $l-2$ cases, we can get
\[
    G(q^{l-1}/(1-q^l))\sim G(q^{l-2}/(1-q^l))^q\Omega_l^{(-1)}\sim G(q^{l-3}/(1-q^l))^{q^2}(\Omega_l^{(-2)})^q\Omega_l^{(-1)}.
\]
Thus combining the cases of $j=1, 2, \ldots, l-1$, we get
\[
    G(q^{l-1}/(1-q^l))\sim G(1/(1-q^l))(\Omega_l^{(-l)})^{q^{l-2}} (\Omega_l^{(-(l-1)})^{q^{l-3}}\cdots\Omega_l^{(-1)}.
\]
By using Proposition \ref{prop:periodpart},
\[
    G(q^{l-1}/(1-q^l))\sim G(q^{l-1}/(1-q^l))^{q^l}(\Omega_l^{(-l)})^{q^{l-2}} (\Omega_l^{(-(l-1)})^{q^{l-3}}\cdots\Omega_l^{(-1)}.
\]
Thus we have
\[
 G(q^{l-1}/(1-q^l))\sim \Bigl((\Omega_l^{(-l)})^{q^{l-2}} (\Omega_l^{(-(l-1)})^{q^{l-3}}\cdots\Omega_l^{(-1)}\Bigr)^{\frac{1}{1-q^l}}.
\]
In the similar way, for $0\leq j\leq l-1$, we have
$$
    G(q^j/(q^l-1))\sim\Bigl(\Omega_l^{(-(l-j))}\cdots(\Omega_l^{(-(l-1))})^{q^{j-1}}\Omega_l^{q^j}(\Omega_l^{(-1)})^{q^{j+1}}\cdots(\Omega_l^{(-(l-j-1))})^{q^{l-1}}\Bigr)^{\frac{1}{1-q^l}}.
$$
\end{eg}

The above implies
\begin{align*}
&\trdeg_{\overline{k}(t)}\overline{k}(t)\Bigl(\Omega_l^{(-j)}\ |\ 0\leq j\leq l-1\Bigr)=\trdeg_{\overline{k}(t)}\overline{k}(t)\Bigl(G(q^j/(1-q^l))\ |\ 0\leq j\leq l-1\Bigr).
\end{align*}
Therefore we have
\begin{align}\label{hyptran:andersonthakur:gamma:t}
&\trdeg_{\overline{k}(t)}\overline{k}(t)\Bigl(\D^{(i)}_t(\Omega_l^{(-j)})\ |\ 0\leq i\leq n, 0\leq j\leq l-1\Bigr)\\
&=\trdeg_{\overline{k}(t)}\overline{k}(t)\Bigl(\D^{(i)}_t\bigl(G(q^j/(1-q^l))\bigr)\ |\ 0\leq i\leq n, 0\leq j\leq l-1\Bigr).\nonumber
\end{align}

Here, we recall the hyperderivative at $\theta$.
\begin{defn}
For each $n\geq 0$, we define $\mathbb{F}_q$-linear operator $\D_\theta^{n}$ on $k_{\infty}$ by
\begin{align*}
\D_{\theta}^{(n)}:k_{\infty}&\rightarrow k_{\infty}\\
                \sum_{i=m}^{\infty}a_i\theta^{-i}&\mapsto \sum_{i=m}^{\infty}\binom{-i}{n}a_i\theta^{-i-n}.
    \end{align*} 
\end{defn}
The generalized Leibnitz rule \eqref{Leibniz} also holds for $\D_{\theta}^{(n)}$. Further, $\D_{\theta}$ is also defined on $k_{\infty}^{sep}$ (cf. \cite[Example 1.5]{M22}).

By using $\D_{\theta}$, we set the homomorphism (cf. \cite[Definition 1.7]{M22})
\begin{equation}
    \rho_{n, \theta}:\operatorname{Mat}_{d\times d}\bigl(k_\infty^{sep} \bigr) \rightarrow \operatorname{Mat}_{d(n+1)\times d(n+1)}\bigl(k_\infty^{sep} \bigr)
\end{equation}
    defined by
\begin{equation}
    \rho_{n, \theta}(X):=
    \begin{pmatrix}
   X & 0 &\cdots&\cdots&0\\
   \partial^{(1)}_{\theta}X & X&0&\cdots&0\\
   \vdots& \vdots&\ddots&&\vdots \\
   \partial^{(n-1)}_{\theta}X & \partial^{(n-2)}_{\theta}X&\cdots&X&0\\
   \partial^{(n)}_{\theta}X &\partial^{(n-1)}_{\theta}X &\cdots&\partial^{(1)}_{\theta}X &X
\end{pmatrix}.
\end{equation}

In the following lemma, we set $k[t]_{(t-\theta)}$ the localization of $k[t]$ at the prime ideal $(t-\theta)$, that is, rational functions regular at $t=\theta$.

\begin{lem}\label{lem;derivegamma}
For each $j\geq 0$, there is $b_j\in k[t]_{(t-\theta)}$ so that 
\[
    \D_{\theta}^{(j)}\biggl( G\Bigl(\frac{q^i}{1-q^l}\Bigr) \biggr)=b_j G\Bigl( \frac{q^i}{1-q^l} \Bigr).
\]
\end{lem}
\begin{proof}
 When $j=0$, clearly the claim holds by taking $b_0=1$.
 
 We give a proof for $j>0$ case. 
 By generalized Leibniz rule, we have
 \begin{align*}
    &\D_{\theta}^{(j)}\biggl( G\biggl(\frac{q^i}{1-q^l}\biggr) \biggr)=\D_{\theta}^{(j)}\biggl\{ \prod_{i\geq 1}\prod_{n\geq 0}^{il-1}\biggl(  1-\frac{t^{q^n}}{\theta^{q^{il}}}  \biggr) \biggr\} \\
    &=\D_{\theta}^{(j)}\biggl\{ 
    \prod_{N\geq i\geq 1}\biggl(\prod_{n\geq 0}^{il-1}\Bigl(  1-\frac{t^{q^n}}{\theta^{q^{il}}}  \Bigr)\biggr)
    \prod_{i\geq N}\biggl(\prod_{n\geq 0}^{il-1}\Bigl(  1-\frac{t^{q^n}}{\theta^{q^{il}}}  \Bigr) \biggr)\biggr\}\\
    &=\sum_{j^\prime=0}^j\D_{\theta}^{(j-j^\prime)}\biggl\{ 
    \prod_{N> i\geq 1}\biggl(\prod_{n\geq 0}^{il-1}\Bigl(  1-\frac{t^{q^n}}{\theta^{q^{il}}}  \Bigr)\biggr)\biggr\}
    \D_{\theta}^{(j^\prime)}\biggl\{\prod_{i> N}\biggl(\prod_{n\geq 0}^{il-1}\Bigl(  1-\frac{t^{q^n}}{\theta^{q^{il}}}  \Bigr) \biggr)\biggr\}
 \end{align*}
for any $N \geq 1$.
Recall here that for given $h\in\mathbb{N}$,
  we have $\D_{\theta}^{(n)}(f^{q^h})=0$ for all $f\in k_{\infty}^{sep}$ and all $n$ such that $q^h>n>0$ (cf. \cite[Proof of Lemma 2.2]{M22}).
 Hence, if $N$ is taken so that $q^{Nl}>j$, we have $\D_{\theta}^{(j^\prime)}(\theta^{-q^{il}})=0$ for $i>N$ and $0<j^\prime  \leq j$, thus it holds by generalized Leibniz rule that
\begin{equation}
     \D_{\theta}^{(j^\prime)}\biggl\{\prod_{i> N}\biggl(\prod_{n\geq 0}^{il-1}\Bigl(  1-\frac{t^{q^n}}{\theta^{q^{il}}}  \Bigr) \biggr)\biggr\}=0
 \end{equation}
for $0<j^\prime  \leq j$.
Consequently, we have
 \begin{align*}
    \D_{\theta}^{(j)}\biggl( G\biggl(\frac{q^i}{q^l-1}\biggr) \biggr)
    &=\D_{\theta}^{(j)}\biggl\{ 
    \prod_{N> i\geq 1}\biggl(\prod_{n\geq 0}^{il-1}\Bigl(  1-\frac{t^{q^n}}{\theta^{q^{il}}}  \Bigr)\biggr)\biggr\}
    \biggr\{\prod_{i> N}\biggl(\prod_{n\geq 0}^{il-1}\Bigl(  1-\frac{t^{q^n}}{\theta^{q^{il}}}  \Bigr) \biggr)\biggr\}\\
    &=\D_{\theta}^{(j)}\biggl\{ 
    \prod_{N\geq i\geq 1}\biggl(\prod_{n\geq 0}^{il-1}\Bigl(  1-\frac{t^{q^n}}{\theta^{q^{il}}}  \Bigr)\biggr)\biggr\}
    \biggl\{ \prod_{N\geq i\geq 1}\biggl(\prod_{n\geq 0}^{il-1}\Bigl(  1-\frac{t^{q^n}}{\theta^{q^{il}}}  \Bigr)\biggr) \biggr\}^{-1}G\Bigl( \frac{q^i}{1-q^l} \Bigr)
 \end{align*} 
and thus we can take $b_j$ as
$$
    b_j=\D_{\theta}^{(j)}\biggl\{ 
    \prod_{N\geq i\geq 1}\biggl(\prod_{n\geq 0}^{il-1}\Bigl(  1-\frac{t^{q^n}}{\theta^{q^{il}}}  \Bigr)\biggr)\biggr\}
    \biggl\{ \prod_{N\geq i\geq 1}\biggl(\prod_{n\geq 0}^{il-1}\Bigl(  1-\frac{t^{q^n}}{\theta^{q^{il}}}  \Bigr)\biggr) \biggr\}^{-1}.
$$
\end{proof}

\begin{lem}\label{lem;derivespan}
    For all $n\geq 0$,  we have
\begin{align}\label{claim;derivespan}
    &\D_{\theta}^{(n)}\biggl( \Gamma\biggl( \frac{q^h}{q^l-1} \biggr) \biggr)
    +\D_{t}^{(n)}\biggl( G\biggl( \frac{q^h}{1-q^l} \biggr)\biggr)|_{t=\theta}\\
    &\quad \in\Span_{k}\biggl\{ \D_t^{(j)}\biggl( G\Bigl(\frac{q^h}{1-q^l}\Bigr) \biggr)|_{t=\theta}\ |\ 0\leq j\leq n-1  \biggr\}.
\end{align}
\end{lem}
\begin{proof}
    By using \cite[Lemma 1.6]{M22},
    \begin{align*}
        &-\D_{\theta}^{(n)}\biggl\{ \Gamma\biggl( \frac{q^h}{1-q^l} \biggr)\biggr\}=\sum_{i+j=n}\biggl( \D_{t}^{(i)}\D_{\theta}^{(j)}\Bigl\{ G\Bigl(\frac{q^h}{1-q^l}\Bigr) \Bigr\} \biggr)|_{t=\theta} \\
        \intertext{then, by Lemma \ref{lem;derivegamma},}
        &\quad =\sum_{i+j=n}\D_t^{i}\biggr( b_jG(\frac{q^h}{1-q^l}) \biggr)|_{t=\theta}\\
        &\quad =\biggl( \sum_{i+j=n}\sum_{i_1+i_2=i}\D_t^{(i_1)}(b_j)\D_t^{(i_2)}\biggl(G\Bigl(\frac{q^h}{1-q^l}\Bigr) \biggr) \biggr)|_{t=\theta}\\
        &\quad =\D_t^{(n)}\biggl( G\Bigl(\frac{q^h}{1-q^l}\Bigr) \biggr)|_{t=\theta}+\sum_{i+j=n}\sum_{\substack{i_1+i_2=i\\i_2<n}}\D_t^{(i_1)}(b_j)|_{t=\theta}\D_t^{(i_2)}\biggl( G\Bigl(\frac{q^h}{1-q^l}\Bigr) \biggr)|_{t=\theta}.
    \end{align*}
   Therefore we get \eqref{claim;derivespan}. 
\end{proof}

\begin{lem}\label{Lem gamma space deform space}
For all $n\geq 0$,
    \begin{align*}
    \Span_{\overline{k}}\biggl\{ \D_{\theta}^{(j)}\biggl(\Gamma\Bigl(\frac{q^h}{1-q^l}\Bigr)\biggr)\ |\ 0\leq j\leq n \biggr\}=\Span_{\overline{k}}\biggl\{ \D_{t}^{(j)}\biggl(G\Bigl(\frac{q^h}{1-q^l}\Bigr)\biggr)|_{t=\theta}\ |\ 0\leq j\leq n \biggr\}.
    \end{align*}
\end{lem}
\begin{proof}
First we prove the following inclusion by the induction on $n$.
\begin{align}\label{proof;supset}
    \Span_{\overline{k}}\biggl\{ \D_{\theta}^{(j)}\biggl(\Gamma\Bigl(\frac{q^h}{1-q^l}\Bigr)\biggr)\ |\ &0\leq j\leq n \biggr\}\\
    &\supset \Span_{\overline{k}}\biggl\{ \D_{t}^{(j)}\biggl(G\Bigl(\frac{q^h}{1-q^l}\Bigr)\biggr)|_{t=\theta}\ |\ 0\leq j\leq n \biggr\}.\nonumber
    \end{align}
    

We have the following by Lemma \ref{lem;derivespan}:
\begin{align*}
    &\D_{\theta}^{(n)}\biggl( \Gamma\biggl( \frac{q^h}{1-q^l} \biggr) \biggr)
    +\D_{t}^{(n)}\biggl( G\biggl( \frac{q^h}{1-q^l} \biggr)\biggr)|_{t=\theta}
    \in\Span_{\overline{k}}\biggl\{ \D_t^{(j)}\biggl( G\Bigl(\frac{q^h}{1-q^l}\Bigr) \biggr)|_{t=\theta}\ |\ 0\leq j\leq n-1  \biggr\}.
\end{align*}
By the induction hypothesis, 
\[
    \Span_{\overline{k}}\biggl\{ \D_{\theta}^{(j)}\biggl(\Gamma\Bigl(\frac{q^h}{1-q^l}\Bigr)\biggr)\ |\ 0\leq j\leq n-1 \biggr\}\supset \Span_{\overline{k}}\biggl\{ \D_t^{(j)}\biggl( G\Bigl(\frac{q^h}{1-q^l}\Bigr) \biggr)|_{t=\theta}\ |\ 0\leq j\leq n-1  \biggr\}
\]
and thus 
\begin{align*}
    &\D_{t}^{(n)}\biggl( G\biggl( \frac{q^h}{1-q^l} \biggr)\biggr)|_{t=\theta}
    \in\Span_{\overline{k}}\biggl\{ \D_t^{(j)}\biggl( G\Bigl(\frac{q^h}{1-q^l}\Bigr) \biggr)|_{t=\theta}\ |\ 0\leq j\leq n-1  \biggr\}.
\end{align*}
This shows the inclusion \eqref{proof;supset}. 

The following inclusion is given by Lemma \ref{lem;derivespan}:    
\begin{align}\label{proof;subset}
    &\Span_{\overline{k}}\biggl\{ \D_{\theta}^{(j)}\biggl(\Gamma\Bigl(\frac{q^h}{1-q^l}\Bigr)\biggr)\ |\ 0\leq j\leq n \biggr\}\\
    &\subset \Span_{\overline{k}}\biggl\{ \D_{t}^{(j)}\biggl(G\Bigl(\frac{q^h}{1-q^l}\Bigr)\biggr)|_{t=\theta}\ |\ 0\leq j\leq n \biggr\}.
    \end{align}

Therefore we obtain the desired equality by combining \eqref{proof;supset} and \eqref{proof;subset}.
\end{proof}
The above Lemma \ref{Lem gamma space deform space} implies that
\begin{align}\label{hyptran:andersonthakur:gamma}
&\overline{k}\biggl(\D_t^{(i)}\Bigl(G\Bigl(\frac{q^j}{1-q^l}\Bigr)\Bigr)|_{t=\theta}\ |\ 0\leq i\leq n, 0\leq j\leq l-1  \biggr)\\
&=\overline{k}\biggl(\D_{\theta}^{(i)}\Bigl(\Gamma\Bigl(\frac{q^j}{1-q^l}\Bigr)\Bigr)\ |\ 0\leq i\leq n, 0\leq j\leq l-1  \biggr).\nonumber
\end{align}

We recall that the Drinfeld module and its associated pre-$t$-motive given in \cite{CPTY10}.

\begin{defn}
Let $k\subset K\subset \mathbb{C}_{\infty}$ and $r\geq 1$. For $\alpha_1, \ldots, \alpha_r\in K$ with $\alpha_r\neq 0$, the Drinfeld $\mathbb{F}_q[t]$-module of rank $r$ 
is the $\mathbb{F}_q$-linear ring homomorphism $\rho$ given by
\begin{align*}
	\rho:\mathbb{F}_q[t]&\rightarrow K[\tau] \\
	           t &\mapsto \theta +\alpha_1\tau+\cdots +\alpha_r\tau^r.   
\end{align*}

\end{defn}

We define a field $K_{\rho}$ by the fraction field of $\End(\rho)=\{ a\in\mathbb{C}_{\infty}\ |\ a\Lambda_{\rho}\subseteq \Lambda_{\rho} \}$.

Then Namoijam's result is stated as follows:
\begin{thm}{\cite[Theorem 1.1.3]{N21}}\label{Namoijamformula}
Let $\rho$ be a Drinfeld $\mathbb{F}_q[t]$-module of rank $r$ defined over $k^{sep}$ and suppose that $K_{\rho}$ is separable over $k$.
Let $P_{\rho}$ be a period matrix of $\rho$, which accounts all periods and quasi-periods.
If $s=[K_{\rho}:k]$, then for $n\geq 1$, 
\[
\trdeg_{\ol{k}}\ol{k}\bigl( \D_{\theta}^{i}(P_{\rho})\ |\ 0\leq i\leq n \bigr)=(n+1)\cdot r^2/s.
\]
\end{thm}
Here $\D_{\theta}^{i}(P_{\rho})$ is formed by entry-wise action of $\D_{\theta}^{i}$ on $P_{\rho}$

\begin{defn}
    For a fixed $l\in\mathbb{Z}_{>0}$ and $k\subset K\subset \mathbb{C}_{\infty}$ with $\mathbb{F}_{q^l}\subset K$, the Carlitz $\mathbb{F}_{q^l}[t]$-module is the following $\mathbb{F}_{q^l}$-linear ring homomorphism:
    \begin{align*}
        C_l:\mathbb{F}_{q^l}[t]&\rightarrow K[\tau]\\
                            t&\mapsto \theta+\tau^l.
    \end{align*}
\end{defn}  

We note that 
$C_l$ is also regarded as rank $l$ Drinfeld module over $\mathbb{F}_q[t]$. In \cite{CPTY10}, they described its associated pre-$t$-motive $M_l$, whose $\sigma$-action is represented by the following $\Phi_l$:
\begin{align}\label{Phi}
\Phi_l:=
&\begin{array}{rccccccll}
\ldelim({5}{1pt}[] &0 & 1 & 0 & \cdots & 0 & \rdelim){5}{1pt}[] &\\
&0 & 0 & 1 & \cdots & 0 & & \\
&\vdots &    \vdots    & \ddots & \ddots & \vdots  &  & \in{\rm GL}_{l}(\overline{k}(t))\cap\Mat_l(\overline{k}[t]). \\
&0      & 0 & \cdots & 0 & 1 & &\\
&(t-\theta) & 0 & 0 & \cdots   & 0 &
\end{array}\\\nonumber
\end{align}
They also gave the rigid analytic trivialization of $\Phi_l$: 
let $\xi_l$ be a primitive element of $\mathbb{F}_{q^l}$ and let $\Psi_l$ be
\begin{align}\label{Psi}
\Psi_l:=
&\begin{array}{rccccccll}
\ldelim({5}{1pt}[] &\Omega_l & \xi_l\Omega_l & \cdots & \xi_l^{l-1}\Omega_l   & \rdelim){5}{1pt}[] &\\
&\Omega_l^{(-1)} & (\xi_l\Omega_l)^{(-1)} & \cdots & (\xi_l^{l-1}\Omega_l)^{(-1)}  & & \\
&\vdots          &    \vdots              & \ddots & \vdots   &  & \in{\rm GL}_{l}(\mathbb{L})\cap\Mat_l(\mathbb{T}). \\
&\Omega_l^{(-(l-1))}      & (\xi_l\Omega_l)^{(-(l-1))} & \cdots & (\xi_l^{l-1}\Omega_l)^{(-(l-1))}  & &
\end{array}\\\nonumber
\end{align}

We note that ${\Psi_l}^{-1}|_{t=\theta}$ is a period matrix of Drinfeld $\mathbb{F}_q[t]$-module $C_l$ (cf. \cite[Theorem 3.4.7]{NP22}).

Because $C_l$ has complex multiplication by $\mb{F}_{q^l}[\theta]$ (cf. \cite[pp.378--379]{G96}) , we have $K_{C_l}=\mb{F}_{q^l}(\theta)$ and $[K_{C_l}:k]=l$. Thus by Theorem \ref{Namoijamformula}, we obtain the following:
$$
\trdeg_{\ol{k}}\ol{k}\bigl( \rho_{n, \theta}({\Psi_l|_{t=\theta}}^{-1})\ |\ 0\leq i\leq n \bigr)=l(n+1).
$$
On the otherhand, $\ol{k}\bigl( \rho_{n, \theta}({\Psi_l|_{t=\theta}}^{-1}) \bigr)=\ol{k}\bigl( \{\rho_{n, \theta}({\Psi_l|_{t=\theta}})\}^{-1} \bigr)=\ol{k}\bigl( \rho_{n, \theta}({\Psi_l|_{t=\theta}} ) \bigr)$. Here the first equality holds since $\rho_{n,\,\theta}$ is a ring homomorphism.
Thus we have 
\begin{thm}\label{ThmonGamma}For $n \geq 0$ and $l \geq 1$,  
    \begin{align*}
    &\trdeg_{\ol{k}}\ol{k}\bigl( \D_{\theta}^{(i)}\bigl(\Gamma(q^j/(1-q^l))\bigr)\ |\ 0\leq i\leq n, 0\leq j\leq l-1 \bigr)\\
&=\trdeg_{\ol{k}}\ol{k}\bigl( \D_{\theta}^{(i)}(\Omega_l^{-(l-j)}|_{t=\theta})\ |\ 0\leq i\leq n, 0\leq j\leq l-1 \bigr)=l(n+1).
\end{align*}

\end{thm}

By using \eqref{hyptran:andersonthakur:gamma}, we also obtain

\begin{equation}\label{hyptran:andersonthakur:theta}
\trdeg_{\ol{k}}\overline{k}\biggl(\D_t^{(i)}\Bigl(G\Bigl(\frac{q^j}{1-q^l}\Bigr)\Bigr)|_{t=\theta}\ |\ 0\leq i\leq n, 0\leq j\leq l-1  \biggr)=l(n+1)
\end{equation}

By applying Theorem \ref{RefinedPapanikolasThm} (ii) to the above and by \eqref{hyptran:andersonthakur:gamma:t}, we further get 
\begin{align}\label{hyptran:omega:and:G}
&\trdeg_{\overline{k}(t)}\overline{k}(t)\Bigl(\D^{i}_t(\Omega_l^{(-j)})\ |\ 0\leq i\leq n, 0\leq j\leq l-1\Bigr)\\
&=\trdeg_{\overline{k}(t)}\overline{k}(t)\Bigl(\D^{i}_t\bigl(G(q^j/(1-q^l))\bigr)\ |\ 0\leq i\leq n, 0\leq j\leq l-1\Bigr)=l(n+1).
\end{align}

\section{$t$-motivic Galois group of pre-$t$-motives associated to arithmetic gamma  values and their prolongations}

    In this section, we consider the prolongations of the pre-$t$-motives $C \oplus M_l$ discussed in \cite{CPTY10}, which have Carlitz period $\tilde{\pi}$ and arithmetic gamma values as their periods. By Theorems \ref{RefinedPapanikolasThm} and \ref{ThmonGamma}, we calculate the dimension of $t$-motivic Galois groups of prolongations of aforementioned pre-$t$-motives.
    
    We fix a non-negative integer $n$ and positive integer $l$.
       Let $N$ be the pre-$t$-motive defined by 
       \begin{equation}
           (t-\theta)\oplus \underbrace{\begin{pmatrix}
               0&1&0&\cdots&0&0\\
               0&0&1&0&\cdots&0\\
               \vdots&\vdots&\ddots&\ddots&\ddots&\vdots\\
               0&0&\cdots&\ddots&1&0\\
               0&0&\cdots&\cdots&0&1\\
               (t-\theta)&0&\cdots&\cdots&0&0\\
           \end{pmatrix}}_{\text{ \normalfont size $l$}}\in\operatorname{Mat}_{l+1}(\overline{k}[t])\cap \operatorname{GL}_{l+1}(\overline{k}(t)).
       \end{equation}
        We rather consider the $l$-th derived pre-$t$-motive $N^{(l)}$ of $N$, instead of $N$ itself. The pre-$t$-motive $N^{(l)}$ is defined by the diagonal matrix $\overline{\Phi}_l$ given by
            \begin{align}\label{DefoverlinePhi}
                \ \ \ \ \ &\left( (t-\theta)^{(-l+1)}\cdots (t-\theta) \right) \oplus\begin{pmatrix}
                    (t-\theta)&0&\cdots&\cdots&0\\
                    0&(t-\theta^{(-1)})&0&\cdots&0\\
                    \vdots&0&(t-\theta^{(-2)})&\ddots &\vdots\\
                    \vdots&\vdots&\ddots&\ddots&0\\
                    0&0&\cdots&0&(t-\theta^{(-l+1)})
                \end{pmatrix}\\
                &\in\operatorname{Mat}_{l+1}(\overline{k}[t])\cap \operatorname{GL}_{l+1}(\overline{k}(t)).
            \end{align}
        We note that the defining matrix $\overline{\Phi}_l$ has a rigid analytic trivialization
                \begin{equation}\label{DefoverlinePsi}
                    \overline{\Psi}_l:=(\Omega) \oplus \begin{pmatrix}
                        \Omega_l&0&\cdots&\cdots&0\\
                        0&\Omega_l^{(-1)}&0&\cdots&0\\
                        \vdots&0&\Omega_l^{(-2)}&\ddots &\vdots\\
                        \vdots&\vdots&\ddots&\ddots&0\\
                        0&0&\cdots&0&\Omega_l^{(-l+1)}
                    \end{pmatrix}.
                \end{equation}
            Namely, we have $\overline{\Psi}_l^{(-l)}=\overline{\Phi}_l\overline{\Psi}_l$.
        
            Let us consider the $t$-motivic Galois group $\Gamma_{\rho_nN^{(l)}}$ of the pre-$t$-motive $\rho_nN^{(l)}$. Take variables $a_j$ and $a_{i,\,j}$ for $0 \leq i \leq l-1$ and $0 \leq j \leq n$.
        We also let $\alpha_j$ be the square diagonal matrix 
            \begin{equation}
                \begin{pmatrix}
                        a_{j}&0&\cdots&\cdots&0\\
                        0&a_{0,\,j}&0&\cdots&0\\
                        \vdots&0&a_{1,\,j}&\ddots &\vdots\\
                        \vdots&\vdots&\ddots&\ddots&0\\
                        0&0&\cdots&0&a_{l-1,\,j}
                    \end{pmatrix}.
                \end{equation}
            We let $G$ be the algebraic group consisting of the invertible (block) matrices of the form
                \begin{equation}
                    \begin{pmatrix}
                           \alpha_0 & 0 &\cdots&0\\
                           \alpha_1 & \alpha_0&\cdots&0\\
                           \vdots& \vdots&\ddots&\vdots \\
                           \alpha_{n}& \alpha_{n-1}&\cdots&\alpha_0
                    \end{pmatrix}.\label{shapeofGammaN}
                \end{equation}
            One can check that $G$ is commutative by its form.
            Indeed, if we take matrices
            \begin{equation}
                A=\begin{pmatrix}
                           \alpha_0 & 0 &\cdots&0\\
                           \alpha_1 & \alpha_0&\cdots&0\\
                           \vdots& \vdots&\ddots&\vdots \\
                           \alpha_{n}& \alpha_{n-1}&\cdots&\alpha_0
                    \end{pmatrix} \text{ and } A^\prime =\begin{pmatrix}
                           \alpha_0^\prime & 0 &\cdots&0\\
                           \alpha_1^\prime & \alpha_0^\prime&\cdots&0\\
                           \vdots& \vdots&\ddots&\vdots \\
                           \alpha_{n}^\prime& \alpha_{n-1}^\prime&\cdots&\alpha_0^\prime
                    \end{pmatrix}
            \end{equation}
            in $G$, where $\alpha_0,\,\dots,\,\alpha_n,\,\alpha_0^\prime,\,\dots,\,\alpha_n^\prime$ are diagonal matrices of size $l+1$, then the $i,\,j$ block of products $AA^\prime$ and $A^\prime A$ are respectively given by 
            \begin{equation}
                \sum_{m=0}^{i-j}\alpha_{i-j-m}\alpha_m^\prime \text{ and } \sum_{m=0}^{i-j}\alpha_{i-j-m}^\prime \alpha_m
            \end{equation}
            for $1 \leq j \leq i \leq n$.
            As $\alpha_0,\,\dots,\,\alpha_n,\,\alpha_0^\prime,\,\dots,\,\alpha_n^\prime$ are diagonal, they commute each other and hence the commutativity of $G$ holds.
            Note that the matrix $\widetilde{(\overline{\Psi}_l)} \in \operatorname{GL}_{(n+1)(l+1)}(\mathbb{L}\otimes \mathbb{L})$ is of the form \eqref{shapeofGammaN}. Here, the variables $a_j$ and $a_{i,\,j}$ respectively correspond to the quantities
                \begin{equation}
                    \sum_{j^\prime =0}^j (\partial^{(j^\prime)}\Omega^{-1})\otimes(\partial^{(j-j^\prime)}\Omega^{})
                    \text{ and }
                    \sum_{j^\prime =0}^j (\partial^{(j^\prime)}(\Omega_l^{-1})^{(-i)})\otimes(\partial^{(j-j^\prime)}\Omega_l^{(-i)}).
                \end{equation}
            Therefore, $\widetilde{(\overline{\Psi}_l)}$ is an element of $\Gamma_{\rho_n N^{(l)}}(\mathbb{L}\otimes \mathbb{L})$ and the $t$-motivic Galois group $\Gamma_{\rho_n N^{(l)}}$ is a closed subgroup of $G$ by definition and hence is commutative.

It follows from Theorem \ref{RefinedPapanikolasThm} that
\begin{equation}
    \dim \Gamma_N=\trdeg_{\overline{k}}\overline{k}\left(\partial_{t}^{(n^\prime)}\Omega|_{t=\theta},\,\partial_{t}^{(n^\prime)}\Omega_l^{(-l^\prime)}|_{t=\theta} \, \middle|\, 0\leq n^\prime \leq n,\,0\leq l^\prime\leq l-1
        \right).
\end{equation}
Considering the hyperderivatives of both sides of the relation $\Omega\sim\Omega_l \Omega_l^{(-1)}\cdots \Omega_l^{(-l+1)}$, we obtain
\begin{align} \label{hyptran:Omega_1:and:omega_l}&\trdeg_{\overline{k}}\Bigl(\overline{k}\left(\partial_{t}^{(n^\prime)}\Omega|_{t=\theta},\,\partial_{t}^{(n^\prime)}\Omega_l^{(-l^\prime)}|_{t=\theta} \, \middle|\, 0\leq n^\prime \leq n,\,0\leq l^\prime\leq l-1
        \right)\Bigr)\\=&\trdeg_{\overline{k}}\Bigl(\overline{k}\left(\partial_{t}^{(n^\prime)}\Omega_l^{(-l^\prime)}|_{t=\theta} \, \middle|\, 0\leq n^\prime \leq n,\,0\leq l^\prime\leq l-1
        \right)\Bigr)=l(n+1),
\end{align}
by \eqref{hyptran:omega:and:G}, hence $\dim \Gamma_N=l(n+1)$.

\section{Pre-$t$-motives associated to multiple zeta values}\label{section:t-motives mzv}

Following \cite{AT09} and \cite{C14}, we recall the period interpretations of MZVs 
by considering their deformations series and constructing appropriate pre-$t$-motives which interpret those deformation series by entries of rigid analytic trivializations.
Their prolongations will be also discussed and we consider $t$-motivic Galois group of direct product of those pre-$t$-motives. 

We first recall the series $\mathscr{L}_{\mathbf{s},\,jl}$ by \cite{AT09}.   
Given a polynomial $u=\sum_ia_it^i\in\ol{k}[t]$, we define $||u||_{\infty}:=\max_i\{ |a_i|_{\infty} \}$.

 For simplicity, we take a sequence of polynomials $u_1,\,u_2,\,\dots,\,\in \overline{k}[t]$ such that
      \begin{equation}\label{polylogconvergecondition1}
                ||u_i||_\infty<| \theta|_\infty^{\frac{i q}{q-1}}.
            \end{equation}
    For example, we can choose $u_i$ as Anderson-Thakur polynomial $H_{i-1}(t)\in A[t]$ (\cite{AT90}).
Then, we define the following series:
\begin{equation}\label{DeftCMPL}
\mathscr{L}_{\mathbf{s},\,jl}=\mathscr{L}_{jl}:=\sum_{i_l>\cdots>i_{j-1}\geq 0}(\Omega^{s_{j-1}}u_{s_{j-1}})^{(i_{j-1})}\cdots (\Omega^{s_l}u_{s_l})^{(i_l)}\in\mathbb{T}
\end{equation}
 for each index $\mathbf{s}=(s_1, \ldots, s_d)\in\mathbb{N}^d$ and $1 \leq l \leq j \leq d+1$ (\cite{AT09} and \cite{C14}).

        For each pair $\hat{\mathbf{s}}=(\mathbf{s},\,m) $ of an index and non-negative integer, we put
            \begin{equation}
                M[\hat{\mathbf{s}}]:=\rho_mM[\mathbf{s}]
            \end{equation}
        where $M[\mathbf{s}]$ is the pre-$t$-motive defined by the matrix
            \begin{equation}
                \begin{pmatrix}
                    (t-\theta)^{s_{1}+\cdots+s_{d}} & 0& \cdots && \\
                    (t-\theta)^{s_{1}+\cdots+s_{d}} u_{s_1}^{(-1)} & (t-\theta)^{s_{2}+\cdots+s_{d}}&0&\cdots&\\
                    &\ddots&\ddots&\ddots& \\
                    & &&(t-\theta)^{s_{d}}& 0\\
                    & & &(t-\theta)^{s_{d}} u_{s_d}^{(-1)} &1
                \end{pmatrix}
            \end{equation}
        which has a rigid analytic trivialization
            \begin{equation}
                \begin{pmatrix}
            \Omega^{s_{1}+\cdots+s_{d}} & & & & & \\
            \Omega^{s_{2}+\cdots+s_{d}}\mathscr{L}_{21} & \Omega^{s_2+\cdots +s_d} &  &  &  & \\
            \vdots & \Omega^{s_3+\cdots +s_d}\mathscr{L}_{32} & \ddots &  &  & \\
            \vdots & \vdots & \ddots & \ddots &  & \\
            \Omega^{s_{d}}\mathscr{L}_{d1} & \Omega^{s_d}\mathscr{L}_{d2} &  & \ddots & \Omega^{s_d} & \\
            \mathscr{L}_{(d+1) 1}  & \mathscr{L}_{(d+1) 2} & \cdots & \cdots & \mathscr{L}_{(d+1) d} & 1
        \end{pmatrix},
            \end{equation}
        here we put $\mathbf{s}=(s_1,\,s_2,\,\dots,\,s_d)$.

        Take a non-negative integer $n$ and a finite set $I$ satisfying $\eqref{subclosed}$. 
        We put $\hat{I}:=I \times \{0,\,1,\,\dots,\,n\}$ and enumerate this set as $\hat{I}=\{\hat{s}_1,\,\hat{s}_2,\,\hat{s}_3\,\dots,\}$ so that we have the following condition:
        \begin{enumerate}
            \item[(En)] 
                Let us take any $\hat{\mathbf{s}}_i=(\mathbf{s},\,m) \in \hat{I}$,  $\mathbf{s}^\prime \in \operatorname{Sub}(\mathbf{s})$, and $m^\prime \leq m$, then we have $j \leq i$ such that $\hat{\mathbf{s}}_j=(\mathbf{s}^\prime,\,m^\prime)$.\label{enumeratecondition}
        \end{enumerate}
        For $0 \leq i \leq \# \hat{I}$, we put
            \begin{equation}\label{DefMi}
                M_i:=\rho_n C \oplus \bigoplus_{0 \leq j \leq i} M[\hat{\mathbf{s}}_j].
            \end{equation}    
\begin{eg}\label{ExampleMi}
    We take distinct positive integers $s, s^\prime$ and put $I=\{(s),\,(s^\prime),\,(s,\, s^\prime)\}$. Consider the case where $n=2$, $u_s=H_{s-1}$, and $u_{s^\prime}=H_{s^\prime-1}$. If we enumerate $\hat{I}$ as $\hat{\mathbf{s}}_1=((s),\,0)$, $\hat{\mathbf{s}}_2=((s^\prime),\,0)$, $\hat{\mathbf{s}}_3=((s,\,s^\prime),\,0)$, $\hat{\mathbf{s}}_4=((s),\,1)$, and so on, then the pre-$t$-motive $M_2$ is defined by
    \begin{equation}
        \Phi_2:=\begin{pmatrix}
            t-\theta&0&0\\
            1&t-\theta&0\\
            0&1&t-\theta
        \end{pmatrix}
        \oplus
        \begin{pmatrix}
            (t-\theta)^s&\\
            (t-\theta)^sH_{s-1}^{(-1)}&1
        \end{pmatrix}
        \oplus
        \begin{pmatrix}
            (t-\theta)^{s^\prime}&\\
            (t-\theta)^{s^\prime}H_{s^\prime-1}^{(-1)}&1
        \end{pmatrix}
    \end{equation}
    which has a rigid analytic trivialization
    \begin{equation}
        \Psi_2:=\begin{pmatrix}
            \Omega&&\\
            \D_t^{(1)}\Omega&\Omega&\\
            \D_t^{(2)}\Omega&\D_t^{(1)}\Omega&\Omega
        \end{pmatrix}
        \oplus
        \begin{pmatrix}
            \Omega^{s}&\\
            \mathscr{L}_{(s),21}&1
        \end{pmatrix}
        \oplus
        \begin{pmatrix}
            \Omega^{s^\prime}&\\
            \mathscr{L}_{(s^\prime),21}&1
        \end{pmatrix}.
    \end{equation}
    A defining matrix of pre-$t$-motive $M_3$ is given by
    \begin{equation}
        \Phi_3:=\Phi_2 \oplus \begin{pmatrix}
            (t-\theta)^{s+s^\prime}&&\\
            (t-\theta)^{s+s^\prime}H_{s-1}^{(-1)}&(t-\theta)^{s^\prime}&\\
            0&(t-\theta)^{s^\prime}H_{s^\prime-1}^{(-1)}&1
        \end{pmatrix},
    \end{equation}
    that have rigid analytic trivialization
    \begin{equation}
        \Psi_3:=\Psi_2 \oplus 
        \begin{pmatrix}
            \Omega^{s+s^\prime}&&\\
            \Omega^{s^\prime}\mathscr{L}_{(s,\,s^\prime),\,21}&\Omega^{s^\prime}&\\
            \mathscr{L}_{(s,\,s^\prime),31}&\mathscr{L}_{(s,\,s^\prime),32}&1
        \end{pmatrix}.
    \end{equation}
    The pre-$t$-motive $M_4$ is represented by the matrix $\Phi_4$ given by
    \begin{equation}
        \Phi_3\oplus
        \begin{pmatrix}
            (t-\theta)^s&0&0&0\\
            (t-\theta)^sH_{s-1}^{(-1)}&1&0&0\\
            \D_t^{(1)}\left((t-\theta)^s\right) &0&(t-\theta)^s&0\\
            \D_t^{(1)}\left((t-\theta)^sH_{s-1}^{(-1)}\right)&0&(t-\theta)^sH_{s-1}^{(-1)}&1
        \end{pmatrix}
    \end{equation}
    with the matrix
    \begin{equation}
        \Psi_3 \oplus
        \begin{pmatrix}
            \Omega^s&0&0&0\\
            \mathscr{L}_{(s),\,21}&1&0&0\\
            \D_t^{(1)}\left(\Omega^s\right) &0&\Omega^s&0\\
            \D_t^{(1)} \left(\mathscr{L}_{(s),\,21}\right) &0&\mathscr{L}_{(s),\,21}&1
        \end{pmatrix}
    \end{equation}
    as its rigid analytic trivialization.
\end{eg}

        The dimension of the $t$-motivic Galois group $\Gamma_{M_i}$ is equal to the transcendence degree of the field
            \begin{equation}
                \overline{k}\bigl(
                    \partial^{(n^\prime)}\Omega|_{t=\theta},\,
                \partial^{(m)}\mathscr{L}_{\mathbf{s},\,(\dep{\mathfrak{s}}+1)1}|_{t=\theta}
                \mid
                0 \leq n^\prime \leq n,\,(\mathbf{s},\,m)=\hat{\mathbf{s}}_j \text{ for }j \leq i\bigr)
            \end{equation}
over $\overline{k}$ by Theorem \ref{RefinedPapanikolasThm} (ii), see \eqref{DeftCMPL}. Therefore we have
            \begin{equation}\label{dimGammasucc}
                \dim \Gamma_{M_{i-1}} \leq \dim \Gamma_{M_i} \leq \dim \Gamma_{M_{i-1}}+1.
            \end{equation}

        As $M_{i-1}$ and $\rho_n C$ are direct summands of the pre-$t$-motive $M_i$, we have surjective homomorphisms $\pi_i:\Gamma_{M_i^{(l)}} \twoheadrightarrow \Gamma_{(\rho_n C)^{(l)}}$ and $\overline{f}_i:\Gamma_{M_i^{(l)}} \twoheadrightarrow \Gamma_{M_{i-1}^{(l)}}$ as in Lemma \ref{LemmaTannakianProj}.
        For $0 \leq i \leq \# \hat{I}$, we put
            \begin{equation}
                U_i:=\ker \left( \Gamma_{M_i^{(l)}} \twoheadrightarrow \Gamma_{(\rho_n C)^{(l)}}\right).
            \end{equation}
         Then $\overline{f}_i$ induces a surjective morphism $f_i: U_i \twoheadrightarrow U_{i-1}$ for each $1 \leq i \leq \# \hat{I}$ and we put
            \begin{equation}\label{DefViker}
                V_i:=\ker \left(f_i: U_i \twoheadrightarrow U_{i-1}\right).
            \end{equation}

         \begin{eg}
            We continue to consider the same case as Example \ref{ExampleMi}. Take variables $a_0$, $a_1$, $a_2$, $x$, $y$, $z$, and $w$. Let $G_2$ be the algebraic subvariety of $\operatorname{GL}_{7/\overline{\mathbf{F}_q(t)}}$ consisting of matrices of the form
            \begin{equation}
                \begin{pmatrix}
                    a_0&0&0\\
                    a_1&a_0&0\\
                    a_2&a_1&a_0
                \end{pmatrix}
                \oplus
                \begin{pmatrix}
                    a_0^s&0\\
                    a_0^s x&1
                \end{pmatrix}
                \oplus
                \begin{pmatrix}
                    a_0^{s^\prime}&0\\
                    a_0^{s^\prime} y&1
                \end{pmatrix}.
            \end{equation}    
            Then the matrix $\widetilde{\Psi}_2$ is $\mathbb{L}\otimes_{\overline{k}(t)} \mathbb{L}$-valued point of $G_2$ and hence the $t$-motivic Galois group $\Gamma_{M_2}$ is a closed subvariety of $G_2$. Similarly, the $t$-motivic Galois group $\Gamma_{M_3}$ is contained in $G_3$ consisting of matrices of the form
            \begin{equation}
                \begin{pmatrix}
                    a_0&0&0\\
                    a_1&a_0&0\\
                    a_2&a_1&a_0
                \end{pmatrix}
                \oplus
                \begin{pmatrix}
                    a_0^s&0\\
                    a_0^s x&1
                \end{pmatrix}
                \oplus
                \begin{pmatrix}
                    a_0^{s^\prime}&0\\
                    a_0^{s^\prime} y&1
                \end{pmatrix}
                \oplus
                \begin{pmatrix}
                    a_0^{s+s^\prime}&&\\
                    a_0^{s+s^\prime} x&a_0^{s^\prime}&\\
                    a_0^{s+s^\prime}z&a_0^{s^\prime} y&1
                \end{pmatrix},
            \end{equation}
            and hence all elements of $U_3$ are of the form
            \begin{equation}
                I_3
                \oplus
                \begin{pmatrix}
                    1&0\\
                    x&1
                \end{pmatrix}
                \oplus
                \begin{pmatrix}
                    1&0\\
                    y&1
                \end{pmatrix}
                \oplus
                \begin{pmatrix}
                    1&&\\
                     x&1&\\
                    z& y&1
                \end{pmatrix}.
            \end{equation}
            The homomorphism $\overline{f}_3$ is described as follows by Lemma \ref{LemmaTannakianProj},
            \begin{align}
                &\begin{pmatrix}
                    a_0&0&0\\
                    a_1&a_0&0\\
                    a_2&a_1&a_0
                \end{pmatrix}
                \oplus
                \begin{pmatrix}
                    a_0^s&0\\
                    a_0^s x&1
                \end{pmatrix}
                \oplus
                \begin{pmatrix}
                    a_0^{s^\prime}&0\\
                    a_0^{s^\prime} y&1
                \end{pmatrix}
                \oplus
                \begin{pmatrix}
                    a_0^{s+s^\prime}&&\\
                    a_0^{s+s^\prime} x&a_0^{s^\prime}&\\
                    a_0^{s+s^\prime}z&a_0^{s^\prime} y&1
                \end{pmatrix}\\
                &\mapsto 
                \begin{pmatrix}
                    a_0&0&0\\
                    a_1&a_0&0\\
                    a_2&a_1&a_0
                \end{pmatrix}
                \oplus
                \begin{pmatrix}
                    a_0^s&0\\
                    a_0^s x&1
                \end{pmatrix}
                \oplus
                \begin{pmatrix}
                    a_0^{s^\prime}&0\\
                    a_0^{s^\prime} y&1
                \end{pmatrix}.
            \end{align}
            Therefore, the kernel $V_3$ is a closed subgroup of
            \begin{equation}\label{example V_3}
            \left \{ I_7 \oplus \begin{pmatrix}
                    1&0&0\\
                    0&1&0\\
                    z&0&1
                \end{pmatrix}\,\middle|\, z \in \overline{\mathbb{F}_q(t)} \right\} \simeq \mathbb{G}_a.
        \end{equation}
         \end{eg}
         \begin{eg}
             We continue to use the setting of Example \ref{ExampleMi}. If $G_4$ is the algebraic variety consisting of matrices of the form
            \begin{align}
                 &\begin{pmatrix}
                    a_0&0&0\\
                    a_1&a_0&0\\
                    a_2&a_1&a_0
                \end{pmatrix}
                \oplus
                \begin{pmatrix}
                    a_0^s&0\\
                    a_0^s x&1
                \end{pmatrix}
                \oplus
                \begin{pmatrix}
                    a_0^{s^\prime}&0\\
                    a_0^{s^\prime} y&1
                \end{pmatrix}\\
                &\quad \quad
                \oplus
                \begin{pmatrix}
                    a_0^{s+s^\prime}&&\\
                    a_0^{s+s^\prime} x&a_0^{s^\prime}&\\
                    a_0^{s+s^\prime}z&a_0^{s^\prime} y&1
                \end{pmatrix}
                \oplus
                \begin{pmatrix}
                    a_0^s&0&0&0\\
                    a_0^sx&1&0&0\\
                    s a_0^{s-1}a_1&0&a_0^s&0\\
                    s a_0^{s-1}a_1 x+a_0^sw&0&a_0^sx&1
                \end{pmatrix},
             \end{align}
            then we have an immersion $\Gamma_{M_4} \subset G_4$ and the homomorphism $\overline{f}_4$ given by
             \begin{align}
                 &\begin{pmatrix}
                    a_0&0&0\\
                    a_1&a_0&0\\
                    a_2&a_1&a_0
                \end{pmatrix}
                \oplus
                \begin{pmatrix}
                    a_0^s&0\\
                    a_0^s x&1
                \end{pmatrix}
                \oplus
                \begin{pmatrix}
                    a_0^{s^\prime}&0\\
                    a_0^{s^\prime} y&1
                \end{pmatrix}\\
                &\quad \quad
                \oplus
                \begin{pmatrix}
                    a_0^{s+s^\prime}&&\\
                    a_0^{s+s^\prime} x&a_0^{s^\prime}&\\
                    a_0^{s+s^\prime}z&a_0^{s^\prime} y&1
                \end{pmatrix}
                \oplus
                \begin{pmatrix}
                    a_0^s&0&0&0\\
                    a_0^sx&1&0&0\\
                    s a_0^{s-1}a_1&0&a_0^s&0\\
                    s a_0^{s-1}a_1 x+a_0^sw&0&a_0^sx&1
                \end{pmatrix}\\
                &\mapsto 
                \begin{pmatrix}
                    a_0&0&0\\
                    a_1&a_0&0\\
                    a_2&a_1&a_0
                \end{pmatrix}
                \oplus
                \begin{pmatrix}
                    a_0^s&0\\
                    a_0^s x&1
                \end{pmatrix}
                \oplus
                \begin{pmatrix}
                    a_0^{s^\prime}&0\\
                    a_0^{s^\prime} y&1
                \end{pmatrix}
                \oplus
                \begin{pmatrix}
                    a_0^{s+s^\prime}&&\\
                    a_0^{s+s^\prime} x&a_0^{s^\prime}&\\
                    a_0^{s+s^\prime}z&a_0^{s^\prime} y&1
                \end{pmatrix},
             \end{align}
        (see also Lemma \ref{LemmaTannakianProj}). Hence the kernel $V_4$ is a closed subgroup of 
        \begin{equation}\label{example V_4}
            \left \{ I_{10} \oplus \begin{pmatrix}
                    1&0&0&0\\
                    0&1&0&0\\
                    0&0&1&0\\
                    w&0&0&1
                \end{pmatrix}\,\middle|\, w \in \overline{\mathbb{F}_q(t)} \right\} \simeq \mathbb{G}_a.
        \end{equation}
        Any elements of $U_4(\overline{\mathbb{F}_{q^l}(t)})$ is of the form
        \begin{equation}
             I_3
                \oplus
                \begin{pmatrix}
                    1&0\\
                    x&1
                \end{pmatrix}
                \oplus
                \begin{pmatrix}
                    1&0\\
                    y&1
                \end{pmatrix}
                \oplus
                \begin{pmatrix}
                    1&&\\
                     x&1&\\
                    z& y&1
                \end{pmatrix}
                \oplus
                \begin{pmatrix}
                    1&0&0&0\\
                    x&1&0&0\\
                    0&0&1&0\\
                    w&0&x&1
                \end{pmatrix}.
        \end{equation}
         \end{eg}
         
         The above examples show that $V_3$ and $V_4$ are a closed subvarietis of \eqref{example V_3} and of \eqref{example V_4}, respectively. 
         In general,
         we can deduce by the arguments similar to those in \cite[\S 5.3]{Ma24} that the kernel $V_i $ is a closed subvariety of
        \begin{equation}
            \left \{ I_N \oplus \begin{pmatrix}
        1&0 & \cdots&\cdots& 0 \\
        0& 1&0&\cdots&\vdots\\
        \vdots&0&\ddots&\ddots&\vdots\\
        0&\vdots&\ddots&1&0\\
        x& 0&\cdots&0&1
    \end{pmatrix}\,\middle|\, x \in \overline{\mathbb{F}_q(t)} \right\} \simeq \mathbb{G}_a
        \end{equation}
    for each $1 \leq I \leq \# \hat{I}$ from the assumption \eqref{subclosed} and (En). Here, $N=(n+1)+\sum_{j <i}(m+1)(\dep \mathbf{s})$ where we write $\hat{\mathbf{s}}_j=(\mathbf{s},\,m)$.
\section{A proof of main result}

In this section, we prove Theorem \ref{MainThm}, which contains Theorem \ref{intro main result}.
We consider the pre-$t$-motive $(\rho_n N \oplus M_i)^{(l)}$ of level $l$ and compare its $t$-motivic Galois group with those of $(\rho_n C)^{(l)}$, $M_i^{(l)}$, and $(\rho_n N)^{(l)}$ in order to accomplish the proof. 
The theorem implies an algebraic independence (Corollary \ref{maincor}) of the certain family of special values of the deformation series $G$, $\zeta_A^{\mathrm{AT}}$ and their hyperderivatives.

For simplicity, we continue to take a sequence of polynomials $u_1,\,u_2,\,\dots,\,\in \overline{k}[t]$ such that \eqref{polylogconvergecondition1} holds for each $i \geq 1$, and briefly denote by $\mathscr{L}_{\mathbf{s}}$ the series $\mathscr{L}_{\mathbf{s},\,(d+1)1}$ in \eqref{DeftCMPL} for each index $\mathbf{s}=(s_1,\,\dots,\,s_d)$.
Let us fix a finite set $I$ of indices satisfying \eqref{subclosed} and a enumeration $\{\hat{s}_1,\,\hat{s}_2,\,\hat{s}_3\,\dots,\}$ of $\hat{I}:=I \times \{0,\,1,\,\dots,\,n\}$ satisfying the condition (En). Then we can consider pre-$t$-motive $M_i$ in \eqref{DefMi} for $0 \leq i \leq \# \hat{I}$. 
Further, we fix again a positive integer $l$.

For $0 \leq i \leq \# \hat{I}$, we have surjective homomorphism $\varpi_i:\Gamma_{(\rho_n N \oplus M_i)^{(l)}} \twoheadrightarrow \Gamma_{(\rho_n N)^{(l)}}$ since $(\rho_n N)^{(l)}$ is a direct summand of $(\rho_n N \oplus M_i)^{(l)}$ (see Lemma \ref{LemmaTannakianProj}). We put
            \begin{equation}
                \mathcal{U}_i:=\ker \left(\Gamma_{(\rho_n N \oplus M_i)^{(l)}} \twoheadrightarrow \Gamma_{(\rho_n N)^{(l)}}\right).
            \end{equation}     
    
    Similarly, we have surjective homomorphism $\overline{\varphi}_i :\Gamma_{(\rho_n N \oplus M_i)^{(l)}}\twoheadrightarrow \Gamma_{M_i^{(l)}}$.
    We note that there is following commutative diagram
        \begin{equation}\label{diagramcase2}
                \begin{tikzpicture}[auto]

    \node (1-3) at (-1, 2) {$1$}; \node (22) at (1, 2) {$\mathcal{U}_i$}; \node (42) at (4, 2) {$\Gamma_{(\rho_n N \oplus M_i)^{(l)}}$}; \node (62) at (7.5, 2) {$\Gamma_{(\rho_n N)^{(l)}}$};\node (82) at (9.5, 2) {$1$};
    \node (1-5) at (-1, 0) {$1$}; \node (20) at (1, 0) {$U_i$}; \node (40) at (4, 0) {$\Gamma_{M_i^{(l)}} $}; \node (60) at (7.5, 0) {$\Gamma_{(\rho_n C)^{(l)}}$};\node (80) at (9.5, 0) {$1$};
    \draw[->>] (62) to node {$\psi_i$}(60);
    
    \draw[->>] (42) to node {$\overline{\varphi}_i$}(40);
    \draw[->] (22) to node {$ \varphi_{i}$}(20);

    \draw[{Hooks[right]}->] (22) to node {}(42);
    \draw[{Hooks[right]}->] (20) to node {}(40);

    \draw[->>] (42) to node {$\varpi_i$}(62);
    \draw[->>] (40) to node {$\pi_i$}(60);

    \draw[->] (1-3) to node {}(22);
    \draw[->] (1-5) to node {}(20);
    
    \draw[->] (62) to node {}(82);
    \draw[->] (60) to node {}(80);
    \end{tikzpicture}
            \end{equation}
        for $0 \leq  i \leq \#\hat{I}$.

        \begin{eg}
            Let $l=2$. We take distinct positive integers $s, s^\prime$ and put $I=\{(s),\,(s^\prime),$ $(s,\,s^\prime)\}$. Also we consider the case where $n=2$, $u_s=H_{s-1}$, and $u_{s^\prime}=H_{s^\prime-1}$.
            Then the pre-$t$-motive $\rho_2N\oplus M_3$ is defined by the matrix $\overline{\Phi}_2 \oplus \Phi_3$ (see \eqref{DefoverlinePhi} and Example \ref{ExampleMi}), which has a rigid analytic trivialization $\overline{\Psi}_3 \oplus \Psi_3$, see \eqref{DefoverlinePsi}.
            Therefore, elements of the group $\Gamma_{\rho_2N\oplus M_3}(\overline{\mathbb{F}_{q^l}(t)})=\Gamma_{(\rho_2N\oplus M_3)^{(l)}}(\overline{\mathbb{F}_{q^l}(t)})$ is of the form
            \begin{align}\label{formof3}
                &\begin{pmatrix}
                    a_0&&&&&&&&\\
                    &a_{0,\,0}&&&&&&&\\
                    &&a_{1,\,0}&&&&&&\\
                    a_1&&&a_0&&&&&&\\
                    &a_{0,\,1}&&&a_{0,\,0}&&&&&\\
                    &&a_{1,\,1}&&&a_{1,\,0}&&&\\
                    a_2&&&a_1&&&a_0&&\\
                    &a_{0,\,2}&&&a_{0,\,1}&&&a_{0,\,0}&\\
                    &&a_{1,\,2}&&&a_{1,\,1}&&&a_{1,\,0}
                \end{pmatrix}\\
                &\oplus\begin{pmatrix}
                    a_0&0&0\\
                    a_1&a_0&0\\
                    a_2&a_1&a_0
                \end{pmatrix}
                \oplus
                \begin{pmatrix}
                    a_0^s&0\\
                    a_0^s x&1
                \end{pmatrix}
                \oplus
                \begin{pmatrix}
                    a_0^{s^\prime}&0\\
                    a_0^{s^\prime} y&1
                \end{pmatrix}
                \oplus
                \begin{pmatrix}
                    a_0^{s+s^\prime}&&\\
                    a_0^{s+s^\prime} x&a_0^{s^\prime}&\\
                    a_0^{s+s^\prime}z&a_0^{s^\prime} y&1
                \end{pmatrix}.
            \end{align}
            The homomorphisms $\varpi_3$ maps the matrix \eqref{formof3} to
            \begin{small}
            \begin{equation}
                \begin{pmatrix}
                    a_0&&&&&&&&\\
                    &a_{0,\,0}&&&&&&&\\
                    &&a_{1,\,0}&&&&&&\\
                    a_1&&&a_0&&&&&&\\
                    &a_{0,\,1}&&&a_{0,\,0}&&&&&\\
                    &&a_{1,\,1}&&&a_{1,\,0}&&&\\
                    a_2&&&a_1&&&a_0&&\\
                    &a_{0,\,2}&&&a_{0,\,1}&&&a_{0,\,0}&\\
                    &&a_{1,\,2}&&&a_{1,\,1}&&&a_{1,\,0}
                \end{pmatrix},
            \end{equation}
            \end{small}
            hence the elements of $\mathcal{U}_3$ is of the form
            \begin{equation}\label{form3_2}
                I_9\oplus I_3
                \oplus
                \begin{pmatrix}
                    1&0\\
                    x&1
                \end{pmatrix}
                \oplus
                \begin{pmatrix}
                    1&0\\
                    y&1
                \end{pmatrix}
                \oplus
                \begin{pmatrix}
                    1&&\\
                     x&1&\\
                    z& y&1
                \end{pmatrix}
            \end{equation}
            By the homomorphism $\varphi_3$, the matrix \eqref{form3_2} in $\mathcal{U}_3(\overline{\mathbb{F}_q(t)})$ is mapped to
            \begin{equation}
                I_3
                \oplus
                \begin{pmatrix}
                    1&0\\
                    x&1
                \end{pmatrix}
                \oplus
                \begin{pmatrix}
                    1&0\\
                    y&1
                \end{pmatrix}
                \oplus
                \begin{pmatrix}
                    1&&\\
                     x&1&\\
                    z& y&1
                \end{pmatrix} \in U_3(\overline{\mathbb{F}_q(t)}).
            \end{equation}
            \indent If $\dim \Gamma_{M_3^{(l)}}=\dim \Gamma_{M_2^{(l)}}+1$, then we have 
            $$
            \dim V_3=\dim U_3-\dim U_2=\left(\dim \Gamma_{M_3^{(l)}}-(n+1)\right)-\left(\dim \Gamma_{M_2^{(l)}}-(n+1)\right)=1
            $$ 
            (for $V_3$, see \eqref{DefViker}) and hence $V_3$ is equal to 
            \begin{equation}
               \left\{I_3 \oplus I_2\oplus I_2 \oplus\begin{pmatrix}
                    1&&\\
                     0&1&\\
                    z& 0&1
                \end{pmatrix}\,\middle|\, z \in \overline{\mathbb{F}_q(t)}\right\} \simeq \mathbb{G}_{a/\overline{\mathbb{F}_q(t)}}.
            \end{equation}
            In this case, we can prove that the subgroup $V_3$  is contained in the image of $\varphi_3$.
            Take arbitrary $\xi \in \overline{\mathbb{F}_q(t)}$ and let 
            \begin{equation}
                v_\xi:=I_3 \oplus I_2\oplus I_2 \oplus\begin{pmatrix}
                    1&&\\
                     0&1&\\
                    \xi& 0&1
                \end{pmatrix}
            \end{equation}
            be the corresponding matrix. Since $\overline{\varphi}_3$ is surjective, we have $\gamma \in \Gamma_{(\rho_n \oplus M_3)^{(2)}}(\overline{\mathbb{F}_q(t)})$ such that $\overline{\varphi}_i(\gamma)=v_\xi$.
            Let us take arbitrary $\alpha \in \overline{\mathbb{F}_q(t)}^\times$. Then we can take $\delta \in \Gamma_{(\rho_n \oplus M_3)^{(2)}}(\overline{\mathbb{F}_q(t)})$ so that $\psi_i\circ\varpi_i(\delta)=\pi_3 \circ \overline{\varphi}_i(\delta)=\alpha \cdot I_3$.
            Since $\Gamma_{(\rho_n N)^{(l)}}$ is commutative, we have $\varphi_3(\delta^{-1}\gamma^{-1}\delta \gamma)=\varphi_3(\delta^{-1})\varphi_3(\gamma^{-1})\varphi_3(\delta)\varphi_3( \gamma)=I_9$, so it follows that the commutator $\delta^{-1}\gamma^{-1}\delta \gamma$ is an element of the kernel $\mathcal{U}_3(\overline{\mathbb{F}_q(t)})$.
            Since the image $\overline{\varphi}_3(\delta)$ is of the form
            \begin{equation}
                \begin{pmatrix}
                    \alpha&&\\
                    *&\alpha&\\
                    *&*&\alpha
                \end{pmatrix}
                \oplus
                \begin{pmatrix}
                    \alpha^s&\\
                    *&1
                \end{pmatrix}
                \oplus
                \begin{pmatrix}
                    \alpha^{s^\prime}&\\
                    *&1
                \end{pmatrix}
                \oplus
                \begin{pmatrix}
                    \alpha^{s+s^\prime}&&\\
                    *&\alpha^{s^\prime}&\\
                    *&*&1
                \end{pmatrix},
            \end{equation}
            we have  
            \begin{equation}\label{calexample}
            \overline{\varphi}_3 (\delta^{-1}\gamma^{-1}\delta \gamma)=\varphi_3(\delta^{-1}\gamma^{-1}\delta \gamma)=v_{(\alpha^{s+s^\prime}-1)\xi}=I_3 \oplus I_2\oplus I_2 \oplus\begin{pmatrix}
                    1&&\\
                     0&1&\\
                    (\alpha^{s+s^\prime}-1)\xi& 0&1
                \end{pmatrix},
            \end{equation}
            which implies that $V_3 \subset \im \varphi_3$ as $\xi$ and $\alpha$ are arbitrarily chosen.
        \end{eg}
        
    \begin{lem}\label{LemPhiinj}
        The homomorphism $\varphi_i$ is an injection.
    \end{lem}
    \begin{proof}
       It is enough to verify that the intersection
      $
           \left(\ker \varpi_i \right) \cap \left(\ker \overline{\varphi}_i \right)
       $
       is a trivial algebraic group since $\ker \varphi_i$ is included in $\left(\ker \varpi_i \right) \cap \left(\ker \overline{\varphi}_i \right)$. 
       As we have $(\rho_n N \oplus M_i)^{(l)}=(\rho_n N)^{(l)} \oplus M_i^{(l)}$, there exists an closed immersion
           $\Gamma_{(\rho_n N \oplus M_i)^{(l)}} \hookrightarrow \Gamma_{(\rho_n N)^{(l)}} \times \Gamma_{M_i^{(l)}} 
       $
       and the following commutative diagram:
       \begin{center}
\begin{tikzpicture}[auto]
 \node (33) at (3, 3) {$\Gamma_{(\rho_n N \oplus M_i)^{(l)}}$};
\node (00) at (0, 0) {$\Gamma_{(\rho_n N)^{(l)}}$}; \node (30) at (3, 0) {$\Gamma_{(\rho_n N)^{(l)}} \times \Gamma_{M_i^{(l)}}$}; \node (60) at (6, 0) {$\Gamma_{M_i^{(l)}}$};

\draw[->>] (33) to node {$\varpi_i$}(00);
\draw[->>] (33) to node {$\overline{\varphi}_i$}(60);
\draw[->>] (30) to node {$\operatorname{pr}_1$}(00);
\draw[->>] (30) to node {$\operatorname{pr}_2$}(60);

\draw[{Hooks[right]}->] (33) to node {} (30);
\end{tikzpicture}

\end{center}
(see Lemma \ref{LemmaTannakianProj}).
The intersection of $\ker \mathrm{pr}_1$ and $\ker \mathrm{pr}_2$ is trivial. These kernels of the projections respectively contain kernels $\ker \varpi_i$ and $\ker \overline{\varphi}_i$, hence the intersection $\left(\ker \varpi_i \right) \cap \left(\ker \overline{\varphi}_i \right)$ is also trivial. 
    \end{proof}

    \begin{prop}\label{Prop_on_Galois_grp}
        The equality
        \begin{equation}\label{desiredineq}
                \dim \Gamma_{(\rho_n N \oplus M_i)^{(l)}} = \dim \Gamma_{M_i^{(l)}}+\dim \Gamma_{(\rho_n N)^{(l)}}- \dim \Gamma_{(\rho_n C)^{(l)}}
            \end{equation}
        holds for $0 \leq  i \leq \#\hat{I}$.
    \end{prop}    
    \begin{proof}
        The one-side inequality $\dim \Gamma_{(\rho_n N \oplus M_i)^{(l)}} \leq \dim \Gamma_{M_i^{(l)}}+\dim \Gamma_{(\rho_n N)^{(l)}}- \dim \Gamma_{(\rho_n C)^{(l)}}$ easily follows from Theorem \ref{RefinedPapanikolasThm} and the trivial inequality 
        \begin{align}
        &\trdeg_{\overline{k}}\overline{k}\left(\partial_{t}^{(n^\prime)}\Omega|_{t=\theta},\,\partial_{t}^{(n^\prime)}\Omega_l^{(-l^\prime)}|_{t=\theta},\,\partial_{t}^{(n^\prime)}\zeta_{A}^{\rm{AT}}(\mathbf{s})|_{t=\theta} \, \middle|\, 0\leq n^\prime \leq n,\,0\leq l^\prime\leq l,\,\mathbf{s} \in I
        \right)\\
        \leq &\trdeg_{\overline{k}}\overline{k}\left(\partial_{t}^{(n^\prime)}\Omega|_{t=\theta},\,\partial_{t}^{(n^\prime)}\Omega_l^{(-l^\prime)}|_{t=\theta} \, \middle|\, 0\leq n^\prime \leq n,\,0\leq l^\prime\leq l
        \right)\\
        &+\trdeg_{\overline{k}}\overline{k}\left(\partial_{t}^{(n^\prime)}\Omega|_{t=\theta},\,\partial_{t}^{(n^\prime)}\zeta_{A}^{\rm{AT}}(\mathbf{s})|_{t=\theta} \, \middle|\, 0\leq n^\prime \leq n,\,\mathbf{s} \in I
        \right)\\
        &-\trdeg_{\overline{k}}\overline{k}\left(\partial_{t}^{(n^\prime)}\Omega|_{t=\theta} \, \middle|\, 0\leq n^\prime \leq n
        \right).
    \end{align}
        
        We prove the opposite inequality by induction on $i$.
        In the case $i=0$, we have $M_i=\rho_n C$ and hence we have 
        \begin{equation}   
        \dim \Gamma_{(\rho_n N \oplus M_i)^{(l)}} =\dim \Gamma_{(\rho_n N)^{(l)} } 
        \end{equation}
        as $\rho_n C$ is a sub-pre-$t$-motive of $\rho_n N$. Therefore,  \eqref{desiredineq} holds.
        
        Let us consider the case $i \geq 1$.
        If $\dim \Gamma_{M_i^{(l)}}=\dim \Gamma_{M_{i-1}^{(l)}}$ (see \eqref{dimGammasucc}), we can obtain \eqref{desiredineq} as follows:
            \begin{align}
                \dim \Gamma_{(\rho_n N \oplus M_i)^{(l)}} 
                &\geq \dim \Gamma_{(\rho_n N \oplus M_{i-1})^{(l)}}\\
                &\geq \dim \Gamma_{M_{i-1}^{(l)}}+\dim \Gamma_{(\rho_n N)^{(l)}}- \dim \Gamma_{(\rho_n C)^{(l)}}\\
                &= \dim \Gamma_{M_i^{(l)}}+\dim \Gamma_{(\rho_n N)^{(l)}}- \dim \Gamma_{(\rho_n C)^{(l)}},
            \end{align}
        the second inequality is given by the induction hypothesis. Hence we assume $\dim \Gamma_{M_i^{(l)}}=\dim \Gamma_{M_{i-1}^{(l)}}+1$ (see \eqref{dimGammasucc}) in what follows.
        
        We claim that the inequality
        \begin{equation}\label{keyequation}
                \dim \varphi_i(\mathcal{U}_i) \geq \dim U_i
            \end{equation}
        holds under this assumption.
                We first note that the kernel $V_i$ in \eqref{DefViker} is isomorphic to $\mathbb{G}_a$ as we have observed that $V_i$ is a closed subgroup of $\mathbb{G}_a$ in the end of Section \ref{section:t-motives mzv} and we have assumed that $\dim \Gamma_{M_i^{(l)}}=\dim \Gamma_{M_{i-1}^{(l)}}+1$, which implies $\dim V_i=\dim U_i-\dim U_{i-1}=\dim \Gamma_{M_i^{(l)}}-\dim \Gamma_{M_{i-1}^{(l)}}=1$.
        
        Let us begin by showing the inclusion $V_i \subset \varphi_i(\mathcal{U}_i)$.
        We take arbitrary $x \in \overline{\mathbb{F}_q(t)}$ and let $ v_x \in V_i(\overline{\mathbb{F}_q(t)})$ be the element which is corresponding to $x$ via the isomorphism $V_i \simeq \mathbb{G}_a$. 
        By surjectivity of $\overline{\varphi}_i:\Gamma_{(\rho_n N \oplus M_i)^{(l)}} \twoheadrightarrow \Gamma_{M_i^{(l)}}$, we can take 
                    \begin{equation}
                        \gamma \in \Gamma_{(\rho_n N \oplus M_i)^{(l)}}(\overline{\mathbb{F}_q(t)})
                    \end{equation} 
        such that $\overline{\varphi}_i(\gamma)=v_x$.
        We also take $a \in \overline{\mathbb{F}_q(t)}^{\times}\setminus \overline{\mathbb{F}_q}^{\times}$ and      
        take
            \begin{equation}
                \delta \in \Gamma_{(\rho_n N \oplus M_i)^{(l)}}(\overline{\mathbb{F}_q(t)})
            \end{equation} 
        such that we have
            \begin{equation}
                \pi_i \circ \overline{\varphi}_i (\delta)= \psi_i \circ \varpi_i(\delta)= \begin{pmatrix}
                    a&0&\cdots&\cdots&0\\
                    0&a&0&\cdots&0\\
                    \vdots&0&a&\ddots &\vdots\\
                    \vdots&\vdots&\ddots&\ddots&0\\
                    0&0&\cdots&0&a
                \end{pmatrix}\in \Gamma_{(\rho_n C)^{(l)}}(\overline{\mathbb{F}_q(t)}),
            \end{equation}
            recall here that $\varpi_i$ and $\psi_i$ are suejection.  
            As $\Gamma_{(\rho_n N)^{(l)}}(\overline{\mathbb{F}_q(t)})$ is a commutative group, we have
                \begin{equation}
                    \varpi_i(\delta^{-1}\gamma^{-1}\delta \gamma)=1 \in \Gamma_{(\rho_n N)^{(l)}}(\overline{\mathbb{F}_q(t)}),
                \end{equation}
            hence $\delta^{-1}\gamma^{-1}\delta \gamma$ is an element of the kernel $ \mathcal{U}_i(\overline{\mathbb{F}_q(t)})$.
            On the other hand, we have 
            \begin{equation}
                    \varphi_i(\delta^{-1}\gamma^{-1}\delta \gamma)=v_x^{-1}v_{a^{\operatorname{wt}\mathbf{s}_i}x}=v_{a^{\operatorname{wt}\mathbf{s}_i}x-x}=v_{(a^{\operatorname{wt}\mathbf{s}_i}-1)x}
            \end{equation}
        by the similar calculation to \eqref{calexample}.
        Since $a \in \overline{\mathbb{F}_q(t)}^{\times}\setminus \overline{\mathbb{F}_q}^{\times}$ and $x \in \overline{\mathbb{F}_q(t)}$ are arbitrarily chosen, we have $V_i \subset \varphi_i(\mathcal{U}_i)$.

        We assumed that $\dim \Gamma_{(\rho_n N \oplus M_{i-1})^{(l)}} \geq \dim \Gamma_{M_{i-1}^{(l)}}+\dim \Gamma_{(\rho_n N)^{(l)}}- \dim \Gamma_{(\rho_n C)^{(l)}}$, which is the induction hypothesis. By this assumption we also have 
        \begin{align}\label{Uinductionhyp}
                        \dim \varphi_{i-1}(\mathcal{U}_{i-1})&=\dim \mathcal{U}_{i-1}=\dim \Gamma_{(\rho_n N \oplus M_{i-1} )^{(l)}} - \dim \Gamma_{(\rho_n N)^{(l)}}\\
                        &\geq \dim \Gamma_{M_{i-1}^{(l)}}-\dim \Gamma_{(\rho_n C)^{(l)}}=\dim U_{i-1}.
        \end{align}
        Here the first equality follows from Lemma \ref{LemPhiinj}.
        
        Next we will obtain $\dim \varphi_i(\mathcal{U}_i) \geq \dim U_i$.
        As $(\rho_n N \oplus M_{i-1})^{(l)}$ is a direct summand of $(\rho_n N \oplus M_i)^{(l)}$ for each $1  \leq i \leq \# \hat{I}$, we have a surjective homomorphism $\Gamma_{(\rho_n N \oplus M_i)^{(l)}} \twoheadrightarrow \Gamma_{(\rho_n N \oplus M_{i-1})^{(l)}}$, which induces a surjection $\mathcal{F}_i:\mathcal{U}_i \twoheadrightarrow \mathcal{U}_{i-1}$.
        If we consider the following commutative diagram
                \begin{equation}
\begin{tikzpicture}[auto]
\node (03) at (0, 3) {$\mathcal{U}_{i}$}; \node (33) at (3, 3) {$\mathcal{U}_{i-1}$}; 
\node (00) at (0, 0) {$U_{i}$}; \node (30) at (3, 0) {$U_{i-1}$,}; 
\draw[->>] (03) to node {$\mathcal{F}_i$}(33);
\draw[->>] (00) to node {$f_i$}(30);
\draw[->] (03) to node {$\varphi_{i}$} (00);
\draw[->] (33) to node {$\varphi_{i-1}$} (30);
\end{tikzpicture}
\end{equation}
then we obtain 
                \begin{align}
                    \dim f_i \left(\varphi_i(\mathcal{U}_i)\right)
                    &=\dim (f_i \circ \varphi_i)(\mathcal{U}_i)
                    =\dim (\varphi_{i-1}\circ \mathcal{F}_i)(\mathcal{U}_i)
                    =\dim \left( \varphi_{i-1}(\mathcal{F}_i(\mathcal{U}_i))\right)\\
                    &= \dim \varphi_{i-1}(\mathcal{U}_{i-1}) \geq \dim U_{i-1}
        \end{align}
        by \eqref{Uinductionhyp}. As we have an exact sequences
        \begin{equation}
            1 \rightarrow V_i \hookrightarrow \varphi_i(\mathcal{U}_i) \twoheadrightarrow f_i \left(\varphi_i(\mathcal{U}_i)\right) \twoheadrightarrow 1
        \end{equation}
        and
        \begin{equation}
            1 \rightarrow V_i \hookrightarrow \mathcal{U}_i\twoheadrightarrow U_{i-1} \twoheadrightarrow 1,
        \end{equation}
        we have
            \begin{equation}
                \dim \varphi_i(\mathcal{U}_i)=\dim V_i+\dim f_i \left(\varphi_i(\mathcal{U}_i)\right)\geq \dim V_i+\dim U_{i-1}=\dim U_i.
        \end{equation}
   We can deduce the inequality \eqref{desiredineq} under the assumption $\dim \Gamma_{M_i^{(l)}}=\dim \Gamma_{M_{i-1}^{(l)}}+1$ as follows:
   \begin{align}
            \dim \Gamma_{(\rho_n N \oplus M_i)^{(l)}}&=\dim \mathcal{U}_i+\dim \Gamma_{(\rho_n N)^{(l)}}=\dim \varphi_i(\mathcal{U}_i)+\dim \Gamma_{(\rho_n N)^{(l)}}\\
            &\geq \dim U_i+\dim \Gamma_{(\rho_n N)^{(l)}}\\
            &=\dim \Gamma_{M_i^{(l)}}+\dim \Gamma_{(\rho_n N)^{(l)}}- \dim \Gamma_{(\rho_n C)^{(l)}}.
        \end{align}
        Here the second equality follows from Lemma \ref{LemPhiinj}.
Therefore we complete our proof for \eqref{desiredineq}.
    \end{proof}           

We get the following Theorem \ref{RefinedPapanikolasThm} and Proposition \ref{Prop_on_Galois_grp}.

\begin{thm}\label{MainThm}
For $n\geq 0 $ and $l \geq 0$ and an finite set $I$ of indices such that \eqref{subclosed} holds, we have
    \begin{align}
        &\trdeg_{\overline{k}}\overline{k}\left(\partial_{t}^{(n^\prime)}\Omega|_{t=\theta},\,\partial_{t}^{(n^\prime)}\Omega_l^{(-l^\prime)}|_{t=\theta},\,\partial_{t}^{(n^\prime)}\mathscr{L}_{\mathbf{s}}|_{t=\theta} \, \middle|\, 0\leq n^\prime \leq n,\,0\leq l^\prime< l,\,\mathbf{s} \in I
        \right)\\
        =&\trdeg_{\overline{k}}\overline{k}\left(\partial_{t}^{(n^\prime)}\Omega|_{t=\theta},\,\partial_{t}^{(n^\prime)}\Omega_l^{(-l^\prime)}|_{t=\theta} \, \middle|\, 0\leq n^\prime \leq n,\,0\leq l^\prime< l
        \right)\\
        &+\trdeg_{\overline{k}}\overline{k}\left(\partial_{t}^{(n^\prime)}\Omega|_{t=\theta},\,\partial_{t}^{(n^\prime)}\mathscr{L}_{\mathbf{s}}|_{t=\theta} \, \middle|\, 0\leq n^\prime \leq n,\,\mathbf{s} \in I
        \right)\\
        &-\trdeg_{\overline{k}}\overline{k}\left(\partial_{t}^{(n^\prime)}\Omega|_{t=\theta} \, \middle|\, 0\leq n^\prime \leq n
        \right).
    \end{align}
\end{thm}

In the case where $u_i$ is the Anderson-Thakur polynomial $H_{i-1}$ for each $i \geq 1$, we have Theorem \ref{intro main result}.
When $n=0$ and $I=\{1,\,2,\,\ldots,\,s\}$, this theorem recovers the algebraic independence of arithmetic gamma values and Carlitz zeta values established in \cite{CPTY10}.
That is, Theorem \ref{MainThm} deduce \cite[Theorem 4.2.2]{CPTY10} by \cite[Corollary 3.3.3]{CPTY10} and \cite[Main Theorem]{CY07}.
Moreover, by \eqref{hyptran:andersonthakur:theta} and Theorem \ref{matsuzuki thm},
we get the following corollary.
\begin{cor} \label{maincor}
    If we take an index $\mathbf{s}=(s_1,\,\dots,\,s_d) \in \mathbb{Z}_{\geq 1}^r$ and assume that $p \nmid s_i,\,(q-1) \nmid s_i$ for $1\leq i\leq d$ and $s_1,\,\dots,\,s_d$ are distinct, then the following is an algebraically independent set
\[
\left\{ \begin{matrix}\D_{t}^{(n^\prime)}G(q^{l'}/(1-q^l))|_{t=\theta},\,\\
\D_{t}^{(n^\prime)}\zeta_{A}^{\rm{AT}}(s_{j_1},\,\dots,\,s_{j_{d^\prime}})|_{t=\theta}
\end{matrix}\, \middle|\, 0\leq n^\prime \leq n,\,0\leq l^\prime\leq l-1,\,1\leq j_1<\cdots<j_{d^\prime}\leq d \right\}.
\]
\end{cor}


\section*{Acknowledgements}
The first author is supported by JSPS KAKENHI Grant Number JP22KJ2534. The second author is supported by JSPS KAKENHI Grant Number JP23KJ1079.

\end{document}